\documentclass{conm-p-l}

\usepackage{amssymb}
\usepackage{graphicx}
\usepackage{verbatim}
\usepackage{tikz}
\usepackage{longtable}
\usepackage[normalsize,sc]{caption}

\copyrightinfo{2008}{American Mathematical Society}

\newtheorem{theorem}{Theorem}[section]
\newtheorem{lemma}{Lemma}[section]
\newtheorem{proposition}{Proposition}[section]
\newtheorem{corollary}{Corollary}[section]
\newtheorem{question}{Question}[section]

\theoremstyle{definition}
\newtheorem{definition}{Definition}[section]
\newtheorem{example}{Example}[section]

\theoremstyle{remark}
\newtheorem{remark}{Remark}[section]

\numberwithin{equation}{section}
\numberwithin{figure}{section}
\numberwithin{table}{section}

\DeclareMathOperator{\codim}{codim}

\newcommand{\RRR}{\mathbb{R}}
\newcommand{\CCC}{\mathbb{C}}
\newcommand{\ND}{\mathcal{N}}
\newcommand{\PD}{\mathit{PD}}
\newcommand{\pd}{\mathit{pd}}
\newcommand{\cor}{\mathit{cor}}

\newcommand{\ind}{\mbox{$\perp \kern-5.5pt \perp$}}
\newcommand{\indsub}{{\mbox{\scriptsize$\perp \kern-4.5pt \perp$}}} 

\newcommand{\rel}[1]{\mbox{$\langle \kern-2.5pt \langle$} #1 
                     \mbox{$\rangle \kern-2.5pt \rangle$}}

\tikzstyle{dual}=[circle,draw=blue!50,fill=blue!20,thick,
inner sep=0pt,minimum size=3mm]
\tikzstyle{usual}=[circle,draw=black!50,fill=black!20,thick,
inner sep=0pt,minimum size=3mm]

\begin{document}

\title[Gaussian conditional independence models]{Smoothness of Gaussian
  conditional independence models}

\author[M. Drton]{Mathias Drton}
\address{Department of Statistics, 5734 S.~University Ave, Chicago, IL  60637}
\email{drton@uchicago.edu} \thanks{This material is based upon work
  supported by the National Science Foundation under Grant No. DMS-0746265.
  Mathias Drton was also supported by an Alfred P. Sloan Fellowship.}

\author[H. Xiao]{Han Xiao}
\address{Department of Statistics, 5734 S.~University Ave, Chicago, IL  60637}
\email{xiao@galton.uchicago.edu}
\thanks{}


\date{}

\begin{abstract}
  Conditional independence in a multivariate normal (or Gaussian)
  distribution is characterized by the vanishing of subdeterminants of
  the distribution's covariance matrix. Gaussian conditional
  independence models thus correspond to algebraic subsets of the cone
  of positive definite matrices. For statistical inference in such
  models it is important to know whether or not the model contains
  singularities.  We study this issue in models involving up to four
  random variables.  In particular, we give examples of conditional
  independence relations which, despite being probabilistically
  representable, yield models that non-trivially decompose
  into a finite union of several smooth submodels.
\end{abstract}

\maketitle

\section{Introduction}
\label{sec:introduction}

Conditional independence (CI) is one of the most important notions of
multivariate statistical modelling.  Many popular statistical models
can be thought of as being defined in terms of CI constraints.  For
instance, the popular graphical models are obtained by identifying a
considered set of random variables with the nodes of a graph and
converting separation relations in the graph into CI statements
\cite{lauritzen:1996}.  Despite the use of different graphs and
separation criteria, graphical models present only a small subset of
the models that can be defined using conditional independence
\cite{studeny:2005}.  It is thus of interest to explore to which
extent more general collections of CI constraints may furnish other
well-behaved statistical models.  In this paper we pursue this problem
under the assumption that the considered random vector is Gaussian,
that is, it has a joint multivariate normal distribution.  A precise
formulation of the problem is given in Question~\ref{q:smooth} below.

Let $X=(X_1,\dots,X_m)$ be a Gaussian random vector with mean vector
$\mu$ and covariance matrix $\Sigma$, in symbols,
$X\sim\ND_m(\mu,\Sigma)$.  All covariance matrices appearing in this
paper are tacitly assumed positive definite, in which case $X$ is also
referred to as {\em regular Gaussian}.  We denote the subvector given
by an index set $A\subseteq [m]:=\{1,\dots,m\}$ by $X_A$.  For three
pairwise disjoint index sets $A,B,C \subseteq [m]$, we write $A \ind B
\,|\, C$ to abbreviate the conditional independence of $X_A$ and $X_B$
given $X_C$.  We use concatenation of symbols to
denote unions of index sets, that is, $AB=A\cup B$, and make
no distinction between indices and singleton index sets such that
$i=\{i\}$ and 
$ij=\{i,j\}$.  A general introduction to conditional independence can
be found in \cite{studeny:2005}, but since this paper is solely
concerned with the Gaussian case the reader may also simply treat the
following proposition as a definition.  It states that conditional
independence in a Gaussian random vector is an algebraic constraint on
its covariance matrix.  For a proof see for example
\cite[\S3.1]{drton:2009}.

\begin{proposition}
  \label{prop:cic}
  Let $X \sim \ND_m(\mu,\Sigma)$ be a (regular) Gaussian random vector and
  $A,B,C \subset [m]$ pairwise disjoint index sets.  Then $A\ind B\,|\,C$
  if and only if the submatrix $\Sigma_{AC,BC}$ has rank equal to the
  cardinality of $C$.  Moreover, $A\ind B\,|\,C$ if and only if $i\ind
  j\,|\,C$ for all $i\in A$ and $j\in B$.
\end{proposition} 

The proposition clarifies in particular that one may restrict
attention to pairwise statements $i \ind j \,|\, C$.  We remark that
this is also true for arbitrary (non-Gaussian) random vectors as
it can still be shown that $A\ind B \,|\, C$ if and only if
\begin{equation*}
  \label{pair}
  i\ind j \,|\, D \quad\text{for all}\quad i\in A,\: j\in B,\: C\subseteq
  D\subseteq ABC\setminus ij;
\end{equation*}
see \cite[Lemma 3]{matus:1992}.  Since $i \ind j \,|\, C$ if and only
if $j \ind i \,|\, C$, pairwise statements can also be represented
using an index set couple $ij|C$ that groups a two-element set
$ij\subseteq[m]$ and a conditioning set $C\subseteq[m]\setminus ij$.
Following \cite{matus:2007} we refer to these couples as {\em
  conditional independence couples}.

A {\em conditional independence relation} is a set of CI couples. We
write $\mathcal{R}(m)$ for the maximal relation comprising all
$\binom{m}{2}\cdot 2^{m-2}$ CI couples over the set $[m]$.  A CI
relation $\mathcal{L}\subseteq\mathcal{R}(m)$ determines a {\em
  Gaussian conditional independence model}, namely, the family of all
multivariate normal distributions for which $i\ind j\,|\,C$ whenever
$ij|C\in \mathcal{L}$.  
Since conditional independence constrains only
the covariance matrix of a Gaussian random vector, the Gaussian model
given by $\mathcal{L}$ corresponds to the algebraic subset
\begin{equation}
  \label{eq:Vpd}
  V_\pd(\mathcal{L}) = \big\{\, \Sigma\in \PD_m \::\:
  \det(\Sigma_{iC,jC})=0 \;\text{for all}\; ij|C\in \mathcal{L}\, \big\}
\end{equation}
of the cone of positive definite $m\times m$-matrices, here denoted
by $\PD_m$.


Standard large-sample asymptotic methodology can be applied for statistical
inference in a Gaussian CI model if $V_\pd(\mathcal{L})$ is a smooth
manifold.  However, such techniques may fail under the presence of
singularities \cite{drton:lrt}, which leads to the following natural
question:

\begin{question} 
  \label{q:smooth}
  For which conditional independence relations
  $\mathcal{L}\subseteq\mathcal{R}(m)$ is the associated set
  $V_\pd(\mathcal{L})$ a smooth manifold?
\end{question}

If $m=2$ the question is trivial because the set $V_\pd(\mathcal{L})$ is
either the positive definite cone or the set of diagonal $2\times
2$-matrices.  For $m=3$, smoothness can fail in precisely one well-known
way; compare (\ref{eq:weaktran}) and Proposition~\ref{thm:ijcd} below.

\begin{proposition}
  \label{prop:m=3}
  For $m=3$,
  the sets $V_\pd(\mathcal{L})$ are smooth manifolds unless the
  conditional independence relation $\mathcal{L}$ is equal to $\{ij,\,
  ij|k\}$ for distinct indices $i,j,k$.
\end{proposition}

In this paper we will answer the question for $m= 4$.  Note that there are
$2^{24}=16,777,216$ relations on $[m]=[4]$.  However, two relations may induce
the same Gaussian model.  For instance,
$V_\pd(\mathcal{L})=V_\pd(\mathcal{K})$ for $\mathcal{L}=\{12,13|2\}$ and
$\mathcal{K}=\{12,13,12|3,13|2\}$.  Therefore, we begin our study of
Question~\ref{q:smooth}, by finding all Gaussian CI models for a random
vector of length $m=4$.  In this work we build heavily on the work
\cite{matus:2007} that determines the CI relations that are {\em
  representable} in Gaussian random vector of length $m\le 4$; see
also \cite{simecek:prague,simecek:compstat}.

\begin{definition}
  \label{def:represent}
  A relation $\mathcal{L}$ is {\em representable} if there exists a
  covariance matrix $\Sigma\in\PD_m$ for which $\det(\Sigma_{iC,jC})=0$ if
  and only if $ij|C\in\mathcal{L}$.
\end{definition}

The remainder of this paper is structured as follows.  All Gaussian CI
models for $m=4$ random variables are found in
Section~\ref{sec:model}.  Correlation matrices and helpful methods
from computational algebra are introduced in Section~\ref{sec:algebra}
and used to answer Question~\ref{q:smooth} for $m=4$ in
Section~\ref{sec:singularity}.  The findings are discussed in
Section~\ref{sec:conclusion}.  Appendix~\ref{sec:appendix} lists all
Gaussian CI models and implications for $m=4$.



\section{Gaussian conditional independence models}
\label{sec:model}

As mentioned in the introduction, 
there is a many-to-one relationship between the
relations $\mathcal{L}$ and the sets of covariance matrices
$V_\pd(\mathcal{L})$.  In this section we explore this relationship and
determine all Gaussian CI models on four variables.


\subsection{Complete relations and representable decomposition}
\label{subsec:complete}

Given a set of covariance matrices $W\subset \PD_m$, we can define a
relation as
\[
\mathcal{L}(W) =\big\{\, ij|C\in\mathcal{R}(m) \::\: \det(\Sigma_{iC,jC})=0
\;\text{for all}\; \Sigma\in W\, \big\}.
\]
The operator $\mathcal{L}(\cdot)$ and the operator $V_\pd(\cdot)$, defined in
Section~\ref{sec:introduction}, are both inclusion-reversing.  In other words, if
two relations satisfy $\mathcal{L}\subseteq\mathcal{K}$ then
$V_\pd(\mathcal{L})\supseteq V_\pd(\mathcal{K})$, and if two sets are
ordered by inclusion as $V\subseteq W$ then
$\mathcal{L}(V)\supseteq\mathcal{L}(W)$.  For any relation $\mathcal{L}$,
it holds that $\mathcal{L}\subseteq\mathcal{L}(V_\pd(\mathcal{L}))$.

\begin{definition}
  \label{def:complete}
  A relation $\mathcal{L}$ is {\em complete} if
  $\mathcal{L}=\mathcal{L}(V_\pd(\mathcal{L}))$, that is, if for
  every couple $ij|C\not\in\mathcal{L}$ there exists a covariance matrix
  $\Sigma\in V_\pd(\mathcal{L})$ with $\det(\Sigma_{iC,jC})\not=0$.
\end{definition}

Clearly, there is a 1:1 correspondence between models and complete
relations.  The following result provides a useful decomposition into
representable pieces.

\begin{theorem}
  \label{thm:repr-decomp-variety}
  Every conditional independence model $V_\pd(\mathcal{L})$ has a
  representable decomposition, that is, it can be decomposed as
  \[
  V_\pd(\mathcal{L}) = V_\pd(\mathcal{L}_1)\cup\dots\cup
  V_\pd(\mathcal{L}_k),
  \]
  where $\mathcal{L}_1,\dots,\mathcal{L}_k$ are representable relations.
  The decomposition can be chosen minimal (i.e.,
  $\mathcal{L}_i\not\subseteq \mathcal{L}_j$ for all $i\not= j$), in which
  case the relations $\mathcal{L}_1,\dots,\mathcal{L}_k$ are unique up to
  reordering.
\end{theorem}
\begin{proof}
  Suppose not all models have a representable decomposition.  Choose
  $V_\pd(\mathcal{L})$ to be a model that is inclusion-minimal among those
  without a representable decomposition.  Enlarging the relation if
  necessary, we may assume that $\mathcal{L}$ is complete.  Since
  $\mathcal{L}$ cannot be representable, every matrix $\Sigma\in
  V_\pd(\mathcal{L})$ is in $V_\pd(\mathcal{L}\cup\{ij|C\})$ for some CI
  couple $ij|C\not\in\mathcal{L}$.  Therefore, there exist complete
  relations $\mathcal{K}_1,\dots,\mathcal{K}_{l}$, all proper supersets of
  $\mathcal{L}$, such that
  \[
  V_\pd(\mathcal{L}) = V_\pd(\mathcal{K}_1)\cup\dots\cup
  V_\pd(\mathcal{K}_l).
  \]
  The relation $\mathcal{L}$ being complete, each $V_\pd(\mathcal{K}_j)$ is
  a proper subset of $V_\pd(\mathcal{L})$.  By the inclusion-minimal choice
  of $V_\pd(\mathcal{L})$, each $V_\pd(\mathcal{K}_j)$ has a representable
  decomposition.  This, however, yields a contradiction as combining the
  decompositions of the $V_\pd(\mathcal{K}_j)$ provides a representable
  decomposition of $V_\pd(\mathcal{L})$.
  
  A representable decomposition can be chosen to be minimal by removing
  unnecessary components.  To show uniqueness, suppose that there are two
  distinct minimal representable decompositions
  \begin{align}
    \label{eq:decomp-alter1}
  V_\pd(\mathcal{L}) &= V_\pd(\mathcal{L}_1)\cup\dots\cup
  V_\pd(\mathcal{L}_k)\\
  \intertext{and}
    \label{eq:decomp-alter2}
  V_\pd(\mathcal{L}) &= V_\pd(\mathcal{K}_1)\cup\dots\cup
  V_\pd(\mathcal{K}_l).
  \end{align}
  Then, for each $i\in[k]$, we have
  \[
  V_\pd(\mathcal{L}_i) \;=\; 
  V_\pd(\mathcal{L}_i) \cap V_\pd(\mathcal{L}) \;=\; 
  \bigcup_{j=1}^l \left[\,V_\pd(\mathcal{L}_i) \cap
    V_\pd(\mathcal{K}_j)\,\right].
  \]
  Since $V_\pd(\mathcal{L}_i) \cap
  V_\pd(\mathcal{K}_j)=V_\pd(\mathcal{L}_i\cup\mathcal{K}_j)$ and
  $\mathcal{L}_i$ is representable, it follows that $\mathcal{L}_i =
  \mathcal{L}_i \cup \mathcal{K}_j$ for some $j$, which implies that
  $\mathcal{K}_j\subseteq \mathcal{L}_i$.  Applying the same argument with
  the role of the two decompositions reversed, we obtain that
  $\mathcal{L}_s\subseteq \mathcal{K}_j\subseteq \mathcal{L}_i$.  By
  minimality, $s=i$, and thus, $\mathcal{L}_i=
  \mathcal{K}_j$.  Hence, every $\mathcal{L}_i$ appears in
  (\ref{eq:decomp-alter2}).  It follows that $k\le l$.  Reversing again the
  role of the decomposition, we find that $k=l$ and the
  $\mathcal{K}_j$ are just a permutation of the
  $\mathcal{L}_i$.
\end{proof}

\begin{theorem}
  \label{thm:complete-intersection-represent}
  A relation $\mathcal{L}$ is complete if and only if it is an intersection
  of representable relations.  The representable relations can be chosen to
  yield a representable decomposition of the model $V_\pd(\mathcal{L})$.
\end{theorem}
\begin{proof}
  Suppose a relation $\mathcal{L}$ is the intersection of representable
  relations $\mathcal{L}_1,\dots,\mathcal{L}_k$.  Consider a CI couple
  $ij|C\in\mathcal{L}(V_\pd(\mathcal{L}))$, that is, $ij|C$ holds for all
  covariance matrices in $V_\pd(\mathcal{L})$.  By assumption,
  $\mathcal{L}\subseteq\mathcal{L}_i$ for all $i\in[k]$.  Hence,
  $V_\pd(\mathcal{L}_i)\subseteq V_\pd(\mathcal{L})$ and thus
  $ij|C\in\mathcal{L}(V_\pd(\mathcal{L}_i))$ for all $i\in[k]$.  But
  $\mathcal{L}(V_\pd(\mathcal{L}_i))=\mathcal{L}_i$ because the
  representable relations $\mathcal{L}_i$ are in particular complete.  It
  follows that $ij|C$ is in each relation
  $\mathcal{L}_1,\dots,\mathcal{L}_k$ and thus also in $\mathcal{L}$.
  
  Conversely, let $\mathcal{L}$ be a complete relation.  Let
  $\mathcal{L}_1,\dots,\mathcal{L}_k$ be representable relations that yield
  a representable decomposition of $V_\pd(\mathcal{L})$ as in
  Theorem~\ref{thm:repr-decomp-variety}.  Since
  $V_\pd(\mathcal{L}_i)\subseteq V_\pd(\mathcal{L})$ for each $i\in[k]$, we
  have that 
  \[
  \mathcal{L}\;=\;\mathcal{L}(V_\pd(\mathcal{L}))\;\subseteq\;
  \mathcal{L}(V_\pd(\mathcal{L}_i))\;=\;\mathcal{L}_i.
  \]
  Hence, $\mathcal{L}$ is a subset of the intersection of
  $\mathcal{L}_1,\dots,\mathcal{L}_k$.  Since we may deduce from 
  \[
  V_\pd(\mathcal{L})\;=\;\bigcup_{i=1}^k V_\pd(\mathcal{L}_i) \;\subseteq\;
  V_\pd\left(\bigcap_{i=1}^k \mathcal{L}_i\right)
  \]
  that
  \[
  \bigcap_{i=1}^k \mathcal{L}_i \;\subseteq\;
  \mathcal{L}\left(V_\pd\left(\bigcap_{i=1}^k
      \mathcal{L}_i\right)\right)\;\subseteq\;
  \mathcal{L}(V_\pd(\mathcal{L}))\;=\;\mathcal{L},
  \]
  we have shown that $\mathcal{L}$ is  the intersection of
  $\mathcal{L}_1,\dots,\mathcal{L}_k$.
\end{proof}

\begin{example}
  The following relations are derived from the marginal independence
  statements $1\ind 23$, $2\ind 13$ and $1\ind 234$, respectively:
  \begin{align*}
    \mathcal{L}_1 &= \{12,13,12|3,13|2\},\\
    \mathcal{L}_2 &= \{12,23,12|3,23|1\},\\
    \mathcal{L}_3 &=
    \{12,13,14,12|3,12|4,13|2,13|4,14|2,14|3,12|34,13|24,14|23\}.
  \end{align*}
  All three are representable.  Since $\mathcal{L}_2\cap\mathcal{L}_3$
  is equal to $\mathcal{L}=\{12,12|3\}$, the latter is a complete
  relation.  However, $\mathcal{L}_2$ and $\mathcal{L}_3$ do not yield
  a representable decomposition of $V_\pd(\mathcal{L})$ because
  \[
  V_\pd(\mathcal{L}_2) \cup V_\pd(\mathcal{L}_3) \;\subsetneq\;
  V_\pd(\mathcal{L}).
  \]
  The minimal representable decomposition of $V_\pd(\mathcal{L})$ is
  instead given by $\mathcal{L}_1$ and $\mathcal{L}_2$.
\end{example}

\begin{remark}
  The graphical modelling literature also discusses {\em strong
    completeness}; see e.g.~\cite{levitz:2001}.  A representable
  relation $\mathcal{L}$ is {\em strongly complete} if the covariance
  matrices $\Sigma\in V_\pd(\mathcal{L})$ with
  $\mathcal{L}\not=\mathcal{L}(\{\Sigma\})$ form a lower-dimensional
  subset of $V_\pd(\mathcal{L})$.  For the CI relations appearing in
  graphical modelling, the set $V_\pd(\mathcal{L})$ typically
  possesses a polynomial parametrization.  It follows that
  $V_\pd(\mathcal{L})$ is the intersection of an irreducible algebraic
  variety and the cone $\PD_m$.  Completeness then implies strong
  completeness by general results from algebraic geometry
  \cite{cox:2007}.
\end{remark}

\subsection{All models on four variables}
\label{subsec:all-models}

Call two relations $\mathcal{L}_1$ and $\mathcal{L}_2$ {\em
  equivalent}, if there exists a permutation of the indices in the
ground set $[m]$ that turns $\mathcal{L}_1$ into $\mathcal{L}_2$.  In
\cite{matus:2007} it is shown that for $m=4$, there are 53 equivalence
classes of representable relations.
In this section, we find all Gaussian conditional independence models
for $m=4$ random variables by constructing all complete relations.
The work in this section will lead to the proof of the following
result:

\begin{theorem}
  \label{thm:all-complete}
  There are 101 equivalence classes of complete relations on the set
  $[m]=[4]$.
\end{theorem}

In the introduction, we stated the equality
$V_\pd(\mathcal{L})=V_\pd(\mathcal{K})$ for the relations
$\mathcal{L}=\{12,13|2\}$ and $\mathcal{K}=\{12,13,12|3,13|2\}$ as an
example of two relations inducing the same model.  Alternatively, we
may view this as $\{12,13|2\}$ implying $\{12|3,13\}$.

\begin{definition}
  \label{def:CI-impl}
  A (Gaussian) conditional independence implication is an ordered pair of
  disjoint CI relations $(\mathcal{L}_1,\mathcal{L}_2)$ such that
  $V_\pd(\mathcal{L}_1)=V_\pd(\mathcal{L}_1\cup\mathcal{L}_2)$.  We denote
  the implication as $\mathcal{L}_1 \Rightarrow \mathcal{L}_2$ and say that
  a relation $\mathcal{L}$ satisfies $\mathcal{L}_1 \Rightarrow
  \mathcal{L}_2$, if $\mathcal{L}_1 \subseteq \mathcal{L}$ implies that
  $\mathcal{L}_2 \subseteq \mathcal{L}$.
\end{definition}

\begin{example}
  Let $i,j,k \in [m]$ be distinct indices and $C \subset [m] \setminus
  ijk$. Then the following are Gaussian CI implications:
  \begin{eqnarray}
    \{ij|C,ik|C\} & \Longrightarrow & \{ij|kC,ik|jC\} \label{c3} \\
    \{ij|C,ik|jC\} & \Longrightarrow & \{ik|C,ij|kC\}  \label{c4} \\
    \{ij|kC,ik|jC\} & \Longrightarrow & \{ij|C,ik|C\}.  \label{c5}
  \end{eqnarray}
  Implication (\ref{c3}) follows from the last assertion in
  Proposition~\ref{prop:cic} and
  an implication known as {\em weak union} that holds for all probability
  distributions.  Implication (\ref{c4}) is referred to as {\em
    contraction} and also holds for all probability distributions.  The
  last implication, (\ref{c5}), is known as {\em intersection} and holds
  for many but not all non-Gaussian distributions.  See for instance
  \cite[\S3.1]{drton:2009} for more background.
\end{example}


We now describe how to construct all complete relations by adapting the
approach taken in the construction of all representable relations in
\cite{matus:2007}.  A key concept is the following notion of duality.


\begin{definition}
  \label{def:dual}
  The {\it dual} of a couple $ij|C \in \mathcal{R}(m)$ is the couple
  $ij|\bar C$ where $\bar C = [m]\setminus ijC$.  The {\it dual} of a relation
  $\mathcal{L}$ on $[m]$ is the relation
  \[
  \mathcal{L}^d=\{ ij|\bar C\;:\; ij|C \in \mathcal{L}\}
  \]
  made up of the dual couples of the elements of $\mathcal{L}$.
\end{definition}

\begin{lemma}
  \label{lem:dual}
  For a positive definite matrix $\Sigma$ and two relations $\mathcal{L}$ and
  $\mathcal{K}$:
  \begin{enumerate}
  \item[(i)] $\mathcal{L}(\{\Sigma\})^d=\mathcal{L}(\{\Sigma^{-1}\})$;
  \item[(ii)] $\mathcal{L} \Rightarrow \mathcal{K}$ if and only if
    $\mathcal{L}^d \Rightarrow \mathcal{K}^d$;
  \item[(iii)] $\mathcal{L}$ is complete if and only if $\mathcal{L}^d$
    has this property.
  \end{enumerate}
\end{lemma}
\begin{proof}
  (ii) and (iii) follow readily from (i), which holds since a
  subdeterminant in an invertible matrix is zero if and only if the
  complementary subdeterminant in the matrix inverse is zero; see for
  instance \cite[Lemma 1]{matus:2007}.
\end{proof} 

Any complete relation is in particular a {\em semigaussoid}, where a
semigaussoid is defined to be a relation $\mathcal{L} \subseteq
\mathcal{R}(m)$ that satisfies the CI implications (\ref{c3}), (\ref{c4}),
and (\ref{c5}) for all distinct $i,j,k \in [m]$ and $C \subset [m]
\setminus ijk$.  The {\em separation graphoid} associated with a simple
undirected graph $G$ with the vertex set $[m]$ is the relation
$$
\rel{G} =\{\, ij|C \in \mathcal{R}(m)\::\: \hbox{$C$ separates $i$ and $j$
  in $G$}\, \}.
$$
It is a semigaussoid since it is {ascending} and {transitive},
that is,
\begin{eqnarray*}
ij|C \in \rel{G} & \Longrightarrow & ij|kC \in \rel{G} \\
ij|C \in \rel{G} & \Longrightarrow & ik|C \in \rel{G} \;\hbox{ or }\; jk|C \in \rel{G}
\end{eqnarray*}
for any three distinct indices $i,j,k$ and $C \subset [m] \setminus ijk$.
The next two lemmas can be shown by slightly modifying the proofs of Lemma
2 and Lemma 3 in \cite{matus:2007}.

\begin{lemma}
  \label{thm:dual3}
  The duals of semigaussoids are semigaussoids.
\end{lemma}
\begin{lemma}
  \label{thm:graph}
  For a relation $\mathcal{L} \subset \mathcal{R}(m)$, define $G$ to be the
  graph on $[m]$ with $i$ and $j$ adjacent if and only if $\mathcal{L}$
  does not contain the couple $ij|[m] \setminus ij$.  If $\mathcal{L}$ is a
  semigaussoid then $\rel{G}\subseteq\mathcal{L}$.
\end{lemma}

Call $ij|C$ a $t$-couple if the cardinality of $C$ is $t$.
In order to find all semigaussoids it suffices, by Lemma
\ref{thm:dual3}, to consider only relations with more 2-couples than
0-couples.  There are 11 unlabelled undirected graphs on 4
nodes.
In light of Lemma
\ref{thm:graph}, we may obtain all semigaussoids by using the
following search strategy (based on an analogous strategy in
\cite{matus:2007}):

\begin{enumerate}
\item[{\em Step 1.}]  Starting from each of the 11 separation
  graphoids, add all the possible 0-couples and 1-couples while keeping the
  number of 0-couples smaller than the number of 2-couples.  
\item[{\em Step 2.}]  For each relation obtained in this way check whether
  it is a semigaussoid, and whether it is equivalent to a previously
  discovered semigaussoid.
\item[{\em Step 3.}]  Find the duals of the semigaussoids discovered in Steps 1
  and 2.  Check which new semigaussoids are equivalent to earlier found
  semigaussoids.
\end{enumerate}


Steps 1 and 2 produce 109 semigaussoids.
Figure~\ref{fig:semigaussoids} shows how many of these 109
semigaussoids are associated with each of the separation graphoids.
The saturated relation $\mathcal{R}(m)$, given by the empty graph, is
omitted from the figure.  In step 3 of our search we obtain an
additional 48 semigaussoids.  Hence, there are $109+48=157$
equivalence classes of semigaussoids.


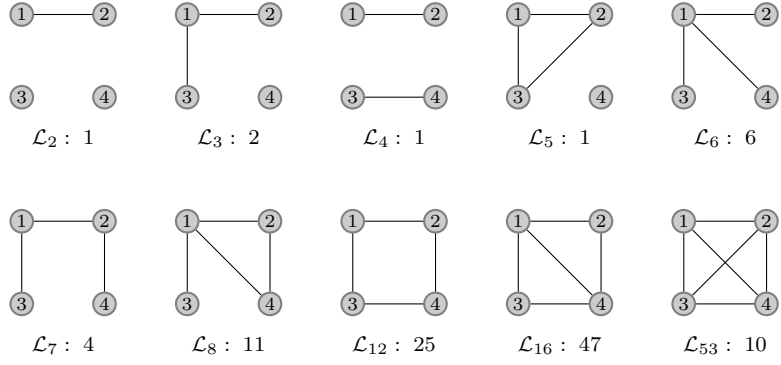
\begin{figure}[t]
  \centering 
  \vspace{0.3cm}
\begin{tikzpicture}[scale=1.1]

\node (1) at ( 0,3.5) [usual] {\scriptsize 1};
\node (2) at ( 1,3.5) [usual] {\scriptsize 2};
\node (3) at ( 0,2.5) [usual] {\scriptsize 3};
\node (4) at ( 1,2.5) [usual] {\scriptsize 4};
\draw [-] (1.east) -- (2.west);

\node (5) at ( 2,3.5) [usual] {\scriptsize 1};
\node (6) at ( 3,3.5) [usual] {\scriptsize 2};
\node (7) at ( 2,2.5) [usual] {\scriptsize 3};
\node (8) at ( 3,2.5) [usual] {\scriptsize 4};
\draw [-] (5.east) -- (6.west);
\draw [-] (5.south) -- (7.north);

\node (9)  at ( 4,3.5) [usual] {\scriptsize 1};
\node (10) at ( 5,3.5) [usual] {\scriptsize 2};
\node (11) at ( 4,2.5) [usual] {\scriptsize 3};
\node (12) at ( 5,2.5) [usual] {\scriptsize 4};
\draw [-] (9.east) -- (10.west);
\draw [-] (11.east) -- (12.west);

\node (13) at ( 6,3.5) [usual] {\scriptsize 1};
\node (14) at ( 7,3.5) [usual] {\scriptsize 2};
\node (15) at ( 6,2.5) [usual] {\scriptsize 3}
    edge [-] (13)
    edge [-] (14);
\node (16) at ( 7,2.5) [usual] {\scriptsize 4};
\draw [-] (13.east) -- (14.west);

\node (18) at ( 9,3.5) [usual] {\scriptsize 2};
\node (19) at ( 8,2.5) [usual] {\scriptsize 3};
\node (20) at ( 9,2.5) [usual] {\scriptsize 4};
\node (17) at ( 8,3.5) [usual] {\scriptsize 1}
    edge [-] (18)
    edge [-] (19)
    edge [-] (20);

\node  at (0.5,2)  {\footnotesize $\mathcal{L}_2:\; 1$};
\node  at (2.5,2)  {\footnotesize $\mathcal{L}_3:\; 2$};
\node  at (4.5,2)  {\footnotesize $\mathcal{L}_4:\; 1$};
\node  at (6.5,2)  {\footnotesize $\mathcal{L}_5:\; 1$};
\node  at (8.5,2)  {\footnotesize $\mathcal{L}_6:\; 6$};


\node (23) at ( 0,0) [usual] {\scriptsize 3};
\node (24) at ( 1,0) [usual] {\scriptsize 4};
\node (22) at ( 1,1) [usual] {\scriptsize 2}
    edge [-] (24);
\node (21) at ( 0,1) [usual] {\scriptsize 1}
    edge [-] (22)
    edge [-] (23);

\node (25) at ( 2,1) [usual] {\scriptsize 1};
\node (26) at ( 3,1) [usual] {\scriptsize 2}
    edge [-] (25);
\node (27) at ( 2,0) [usual] {\scriptsize 3}
    edge [-] (25);
\node (28) at ( 3,0) [usual] {\scriptsize 4}
    edge [-] (25)
    edge [-] (26);

\node (29) at ( 4,1) [usual] {\scriptsize 1};
\node (30) at ( 5,1) [usual] {\scriptsize 2}
    edge [-] (29);
\node (31) at ( 4,0) [usual] {\scriptsize 3}
    edge [-] (29);
\node (32) at ( 5,0) [usual] {\scriptsize 4}
    edge [-] (30)
    edge [-] (31);

\node (33) at ( 6,1) [usual] {\scriptsize 1};
\node (34) at ( 7,1) [usual] {\scriptsize 2}
    edge [-] (33);
\node (35) at ( 6,0) [usual] {\scriptsize 3}
    edge [-] (33);
\node (36) at ( 7,0) [usual] {\scriptsize 4}
    edge [-] (33)
    edge [-] (34)
    edge [-] (35);

\node (37) at ( 8,1) [usual] {\scriptsize 1};
\node (38) at ( 9,1) [usual] {\scriptsize 2}
    edge [-] (37);
\node (39) at ( 8,0) [usual] {\scriptsize 3}
    edge [-] (37)
    edge [-] (38);
\node (40) at ( 9,0) [usual] {\scriptsize 4}
    edge [-] (37)
    edge [-] (38)
    edge [-] (39);

\node  at (0.5,-0.5)  {\footnotesize $\mathcal{L}_7:\; 4$};
\node  at (2.5,-0.5)  {\footnotesize $\mathcal{L}_8:\; 11$};
\node  at (4.5,-0.5)  {\footnotesize $\mathcal{L}_{12}:\; 25$};
\node  at (6.5,-0.5)  {\footnotesize $\mathcal{L}_{16}:\; 47$};
\node  at (8.5,-0.5)  {\footnotesize $\mathcal{L}_{53}:\; 10$};

\end{tikzpicture}
  \caption{Counts of semigaussoids on the 4-element set by
    associated separation graphoid.  The graphoids $\mathcal{L}_i$ are
    labelled in reference to Table~\ref{tab:rep} in
    Appendix~\ref{sec:appendix}.}
  \label{fig:semigaussoids}
\end{figure}

The search for semigaussoids greatly reduces the number of relations.
Among the 157 semigaussoids found above are the 53 representable
relations determined in \cite{matus:2007}, but not all the remaining
104 semigaussoids are complete.  For instance, 10 semigaussoids fail
to satisfy the following CI implications:

\begin{lemma}
  \label{thm:moreci}
  Any complete relation on $[m]$ satisfies
  \begin{eqnarray}
    \{  ij|C ,  kl|C ,  ik|jlC ,  jl|ikC  \}   & \Longrightarrow &  \{ik|C\}
    \label{eq:moreci1} \\ 
    \{  ij|C ,  kl|iC ,  kl|jC ,  ij|klC  \}   & \Longrightarrow &  \{kl|C\}
    \label{eq:moreci2} \\ 
    \{  ij|C ,  jl|kC ,  kl|iC ,  ik|jlC  \}   & \Longrightarrow &  \{ik|C\}
    \label{eq:moreci3} \\ 
    \{          ij|kC ,  ik|lC ,  il|jC   \}   & \Longrightarrow &  \{ij|C\}
    \label{eq:moreci4} \\ 
    \{  ij|kC ,  jk|lC ,  kl|iC ,  il|jC  \}   & \Longrightarrow &  \{ij|C\}
    \label{eq:moreci5}    
  \end{eqnarray}
  for all distinct indices $i,j,k,l$ and $C \subset [m] \setminus ijkl$.
\end{lemma}
\begin{proof}
  These implications are proved in \cite[Lemma 10]{matus:2007}.  In
  Section~\ref{subsec:primdec}, we provide an alternative computer-aided
  proof.
\end{proof}

\begin{proof}[Proof of Theorem~\ref{thm:all-complete}]
  There are 629 representable relations on $[m]=[4]$, when treating
  equivalent but unequal relations as different.  For each relation
  $\mathcal{L}$ among the remaining 94 non-representable semigaussoids
  find all of the 629 representable relations that contain it.  By
  Theorem~\ref{thm:complete-intersection-represent}, $\mathcal{L}$ is
  complete if and only if it is equal to the intersection of these
  representable relations.  We obtain 48 complete relations in
  addition to the representable ones.  This yields the claimed 101
  Gaussian CI models (counting up to equivalence).
\end{proof}

All complete relations on $[m]=[4]$ and their representable
decompositions are listed in the appendix.  One reason for complete
relations to be non-representable is a property known as {\em weak
  transitivity}: For any  matrix $\Sigma\in\PD_m$ it holds
that
\begin{multline}
  \label{eq:weaktran}
  \{ij|C,ij|kC\}\subseteq\mathcal{L}(\{\Sigma\}) \quad\Longrightarrow \\
  \{ik|C, ik|jC\}\subseteq\mathcal{L}(\{\Sigma\}) \;\hbox{ or }\; \{jk|C,
  jk|iC\}\subseteq\mathcal{L}(\{\Sigma\});
\end{multline} 
see for instance~\cite[Ex.~3.1.5]{drton:2009}.  By (\ref{eq:weaktran}), a
representable relation $\mathcal{L}$ satisfies
\begin{equation}
  \label{eq:weaktran-rel}
  \{ij|C,ij|kC\}\subseteq\mathcal{L} \quad\Longrightarrow\quad \{ik|C,
  ik|jC\}\subseteq\mathcal{L} 
  \;\hbox{ or }\; \{jk|C, jk|iC\}\subseteq\mathcal{L}. 
\end{equation} 
Due to the disjunctive conclusion (\ref{eq:weaktran-rel}) is not a CI
implication according to our Definition~\ref{def:CI-impl}.  
The following theorem summarizes results about representable relations
established in \cite{matus:2007}.  


\begin{theorem}
  \label{thm:matus}
  A relation on $[m]=[4]$ is representable if and only if it is a
  semigaussoid that satisfies implications
  (\ref{eq:moreci1})-(\ref{eq:moreci5}) and weak transitivity
  (\ref{eq:weaktran-rel}).
\end{theorem}

To facilitate comparison, we remark that in \cite{matus:2007} a relation
obeying the requirements of a semigaussoid as well as the weak transitivity
property was termed a `gaussoid'.  This motivated choosing the terminology
`semigaussoid' here.

\section{Algebraic techniques}
\label{sec:algebra}

The conditional independence model associated with a relation
$\mathcal{L}\subseteq\mathcal{R}(m)$ corresponds to the algebraic set of
covariance matrices $V_\pd(\mathcal{L})$ defined by the vanishing of
certain `almost-principal' determinants; recall (\ref{eq:Vpd}).  It is thus
natural to begin a study of the geometry of $V_\pd(\mathcal{L})$ by
studying associated ideals of polynomials; see \cite{cox:2007} for some
background.  Before turning to algebraic notions however, we introduce
correlation matrices as a means of reducing later computational effort.

\subsection{Correlation matrices}
\label{subsec:corr}

The {\em correlation matrix} $R=(r_{ij})$ of a (positive definite)
covariance matrix $\Sigma=(\sigma_{ij})$ is the matrix with entries
\[
r_{ij} = \frac{\sigma_{ij}}{\sqrt{\sigma_{ii}\sigma_{jj}}}.
\]
The matrix $R$ is again positive definite, and in particular, $|r_{ij}|<1$
for all $i\not= j$.

\begin{lemma}
  \label{thm:corr}
  Let $R$ be the correlation matrix of $\Sigma\in\PD_m$, and $A,B,C
  \subset [m]$ pairwise disjoint index sets.  Then the conditional
  independence $A\ind B\,|\,C$ holds in $X \sim \ND_m(\mu,\Sigma)$ if and
  only if it holds in $Y\sim\ND_m(0,R)$.
\end{lemma}
\begin{proof}
  Given a CI couple $ij|C$, we have that
  \begin{align*}
    \det(R_{iC,jC})=\frac{1}{\prod_{c\in C}
      \sigma_{rr}}\cdot \frac{1}{\sqrt{\sigma_{ii}\sigma_{jj}}} \cdot
    \det(\Sigma_{iC,jC}),
  \end{align*}
  and thus the claim follows from Proposition~\ref{prop:cic}.
\end{proof}

\begin{example}
  \label{ex:seth}
  Suppose $m\ge 3$ and let $\mathcal{L}$ be the relation given by the
  following pairwise CI statements that each involve three consecutive
  indices (modulo $m$):
  \[
  12|3,\, 23|4,\dots,
  (m-1)m|1,\,1m|2.
  \]
  When stated in terms of the correlation matrix $R=(r_{ij})$, the couple
  $ij|k$ makes the requirement that
  $$
  \det(R_{ik, jk})=r_{ij}-r_{ik}r_{jk}=0.
  $$
  Under the relation $\mathcal{L}$, we thus have
  $r_{i,i+1}=r_{i,i+2}r_{i+1,i+2}$ for all $i\in[m]$, where we take the
  indices modulo $m$. This implies that
  $$
  r_{12}=r_{13}r_{23}=r_{13}r_{24}r_{34}=r_{13}r_{24}r_{35}r_{45}=\cdots
  =\left(\prod_{i=1}^m r_{i,i+2}\right)r_{12}.
  $$
  Since $|r_{i,i+2}|<1$ for all $i \in [m]$, we must have $r_{12}=0$.
  We have thus proved the CI implication $\mathcal{L}\Rightarrow \{12\}$,
  which generalizes the implication (\ref{eq:moreci4}).  
  
  No proper subset $\mathcal{K}\subsetneq\mathcal{L}$ implies $\{12\}$ if
  $m\ge 4$.  This is shown in \cite{sullivant:2009} by a suitable
  counterexample.  We remark that the implication $\mathcal{L}\Rightarrow
  \{12\}$ is also proven in \cite{sullivant:2009} using results on the
  primary decomposition of binomial ideals.  This also sheds light on how
  the implication may fail for singular covariance matrices.
  
  An important feature of this example is that it furnishes an infinite
  family of CI implications that cannot be deduced from other implications.
  It thus establishes that there does not exist a finite set of CI
  implications, from which all other implications can be deduced; compare
  \cite{sullivant:2009,studeny:1992}.
\end{example}

Correlation matrices can also be used to address the smoothness problem
posed in Question~\ref{q:smooth}.  Let $\PD_{m,1}\subset\PD_m$ be the set
of positive definite matrices with ones along the diagonal.  Given a
relation $\mathcal{L}$, we can define the set
\[
V_\cor(\mathcal{L}) = \big\{\, R\in \PD_{m,1} \::\:
\det(R_{iC,jC})=0 \;\text{for all}\; ij|C\in\mathcal{L}\, \big\}.
\]

\begin{lemma}
  \label{lem:R-smooth}
  The model $V_\pd(\mathcal{L})$ is a smooth manifold if and only if
  $V_\cor(\mathcal{L})$ is a smooth manifold.
\end{lemma}
\begin{proof}
  The map that takes a positive definite matrix $\Sigma=(\sigma_{ii})$ as
  argument and returns the vector of diagonal entries
  $(\sigma_{11},\dots,\sigma_{mm})$ and the correlation matrix of $\Sigma$
  is a diffeomorphism $\PD_m\to (0,\infty)^m\times \PD_{m,1}$.
\end{proof}

According to the next fact, we may pass to dual relations when studying the
geometry of $V_\cor(\mathcal{L})$.

\begin{lemma}
  \label{thm:inverse}
  If $\mathcal{L}$ and $\mathcal{L}^d$ are dual relations of each other,
  then $V_\cor(\mathcal{L})$ is diffeomorphic to $V_\cor(\mathcal{L}^d)$.
\end{lemma}
\begin{proof}
  Let $g$ be the map given by matrix inversion and $h$ the map from a
  positive definite matrix to its correlation matrix.  By concatenation, we
  obtain the smooth map $h\circ g:\PD_{m,1}\to\PD_{m,1}$.  This map is its
  own inverse and, thus, $h\circ g:\PD_{m,1}\to\PD_{m,1}$ is a
  diffeomorphism.
  
  By Lemma~\ref{lem:dual}, if $R\in V_\cor(\mathcal{L})$ then
  $g(R)=R^{-1}\in V_\pd(\mathcal{L}^d)$, and the correlation matrix
  $h(R^{-1})$ is in $V_\cor(\mathcal{L}^d)$ according to
  Lemma~\ref{thm:corr}.  Since $(\mathcal{L}^d)^d =\mathcal{L}$, the
  diffeomorphism $h\circ g$ is a bijection between $V_\cor(\mathcal{L})$
  and $V_\cor(\mathcal{L}^d)$.
\end{proof}


\subsection{Conditional independence ideals}
\label{subsec:ci-ideals-saturation}

Let $\mathbb{R}[\mathbf{r}]=\mathbb{R}[r_{ij} \::\: 1 \leq i < j \leq
m]$ be the real polynomial ring associated with the entries $r_{ij}$
of a correlation matrix $R$.  The algebraic geometry of the set
$V_\cor(\mathcal{L})$ is captured by the vanishing ideal
\[
\mathcal{I}(V_\cor(\mathcal{L})) \;=\; \big\{ f \in\mathbb{R}[\mathbf{r}] \::\: f(R)=0 \;\text{for
  all}\; R\in V_\cor(\mathcal{L}) \big\}.
\]
However, it is generally difficult to compute this ideal, where computing
refers to determining a finite generating set.  Instead we start algebraic
computations with the {\em (pairwise) conditional independence ideal}
\[
I_{\mathcal{L}}\;=\; \langle \, \det(R_{iC,jC})\::\: ij|C \in \mathcal{L}
\rangle\;\subseteq\; \mathcal{I}(V_\cor(\mathcal{L})).
\]

\begin{example}
  \label{ex:intersection-3}
  If $\mathcal{L}=\{12|3,13|2\}$ then $I_\mathcal{L} = \langle
  r_{12}-r_{13}r_{23},r_{13}-r_{12}r_{23} \rangle$.  By a simple
  calculation using that $r_{23}^2\not=1$ for correlation matrices, or by
  appealing to the general intersection property (\ref{c5}), we obtain that
  $r_{12}=r_{13}=0$ for all $R\in V_\cor(\mathcal{L})$.  In fact,
  $V_\cor(\mathcal{L})$ is the set of block-diagonal positive definite
  matrices with $r_{12}=r_{13}=0$.  It follows that $\mathcal{I}(V_\cor(\mathcal{L}))=\langle
  r_{12},r_{13}\rangle \not= I_\mathcal{L}$.
\end{example}

\begin{proposition}
  \label{prop:m4-radical}
  Let $\mathcal{L}$ be a relation on $[m]=[4]$.  If $\mathcal{L}$ is
  representable, then $I_\mathcal{L}$ is a radical ideal.  The ideal
  $I_\mathcal{L}$ need not be radical even if $\mathcal{L}$ is complete.
\end{proposition}
\begin{proof}
  We verified the assertion about representable relations by computation of
  all 53 cases with the software package {\tt Singular} \cite{singular}.
  The relation $\mathcal{L}=\{12,14|3,14|23,23|14\}$ is an example of a
  complete relation with $I_\mathcal{L}$  not radical.
\end{proof}

Algebraic calculations with an ideal $I\subset\RRR[\mathbf{r}]$
directly reveal geometric structure of the associated complex
algebraic variety
\[
V_\CCC(I) = \big\{\, R\in \mathbb{S}_{m,1}(\CCC)  \::\:
f(R)=0 \;\text{for all}\; f\in I\, \big\}.
\]
Here, $\mathbb{S}_{m,1}(\CCC)$ is the space of complex symmetric $m\times m$
matrices with ones on the diagonal.  Studying the complex variety
will provide insight into the geometry of the corresponding set of
correlation matrices $V_{cor}(I)$ but, as we will see later, care must
be taken when making this transfer.  

For an ideal $I$ and a polynomial $h$, define the {\em saturation
  ideal}:
\[
(I:h^\infty) = \{\,f \in \RRR[\mathbf{r}]:\; f  h^n \in
I\;\hbox{for some } n \in \mathbb{N}\,\}.
\] 
The variety $V_\CCC(I:h^\infty)$ is the smallest variety containing the set
difference $V_\CCC(I)\setminus V_\CCC(\langle h\rangle)$.  When dealing
with positive definite matrices that have all principal minors positive it
holds that
\[
I_\mathcal{L} \;\subseteq\; (I_\mathcal{L} : D^\infty) \;
\subseteq\; \mathcal{I}(V_\cor(\mathcal{L})),
\]
where $D\in\RRR[\mathbf{r}]$ is the product of all the principal
minors of $R$. Although we have that
$(I_\mathcal{L}:(1-r_{23}^2)^\infty) =
\mathcal{I}(V_\cor(\mathcal{L}))$ in Example~\ref{ex:intersection-3},
saturation with respect to principal minors need not yield the
vanishing ideal $\mathcal{I}(V_\cor(\mathcal{L}))$ in general.  This
occurs for the relations on the left hand side of the implications in
Lemma~\ref{thm:moreci}; saturation with respect to the principal
minors does not change the ideals $I_\mathcal{L}$ considered in the
proof of this lemma in Section~\ref{subsec:primdec}.

If $A=(a_{ij})$ and $B=(b_{ij})$ are matrices in $\PD_m$, then the
Hadamard product $A*B=(a_{ij}b_{ij})$ is a principal submatrix of the
Kronecker product $A\otimes B$.  Hence, $A*B$ is also positive
definite.  As pointed out in \cite{matus:2005}, it can be useful to
consider Hadamard products of $R$ and
$R_\pi=(r_{\pi(i)\pi(j)})$ for permutations $\pi$ on $[m]$ in order to
further enlarge the ideal $I_\mathcal{L}$ by saturation on principal
minors.

\begin{example}
  If $\mathcal{L}$ is the relation from Example~\ref{ex:seth}, then
  $r_{12}$ is seen to be in $I_\mathcal{L}:(1-\prod_{i=1}^m
  r_{i,i+2})^\infty$ and thus in the vanishing ideal
  $\mathcal{I}(V_\cor(\mathcal{L}))$.  The polynomial $1-\prod_{i=1}^m
  r_{i,i+2}$ is a $2\times 2$ minor of a Hadamard product, but of course it
  is also clearly non-zero over $V_\cor(\mathcal{L})$ because each
  $r_{ij}\in(-1,1)$.
\end{example}


However, saturation with respect to `Hadamard product minors' does not seem
to provide the vanishing ideal $\mathcal{I}(V_\cor(\mathcal{L}))$ in
general; compare \cite{matus:2005}.



\subsection{Primary decomposition}
\label{subsec:primdec}

A variety $V_\CCC(I)$ is {\em irreducible} if it cannot be written as
a union of two proper subvarieties of $\mathbb{S}_{m,1}(\CCC)$.  Every
variety has an irreducible decomposition,
\begin{equation}
  \label{eq:irred-decomp}
V_\CCC(I)= V_\CCC(Q_1)\cup\dots\cup V_\CCC(Q_r),
\end{equation}
where the components $V_\CCC(Q_i)$ are irreducible varieties.  The
decomposition is unique up to order when it is minimal, that is, no
component is contained in another; see \cite{cox:2007}.  In that case,
the $V_\CCC(Q_i)$ are referred to as the {\em irreducible components}
of $V_\CCC(I)$.  An irreducible decomposition can be computed by
calculating a {\em primary decomposition} of the ideal $I$, which
writes the ideal as an intersection of so-called primary ideals, $I =
\cap_{i=1}^r Q_i$.  If $I$ is radical then $I$ has an up to order
unique minimal decomposition as an intersection of prime ideals $Q_i$.
Minimality means again that $Q_i\not\subseteq Q_j$ for $i\not= j$.
See again \cite{cox:2007} for the involved algebraic notions.


The computation of primary decompositions of the CI ideals
$I_{\mathcal{L}}$ is in particular useful for investigating CI
implications.  We now show how to use this technique by giving a
computer-aided proof of Lemma~\ref{thm:moreci}.

\begin{proof}[Proof of Lemma~\ref{thm:moreci}]
  By considering the conditional covariance matrix (or Schur complement)
  for $ijkl$ given $C$, it suffices to prove the implications for the case
  $C=\emptyset$.  We may assume $m=4$ and set $i=1$, $j=2$, $k=3$ and
  $l=4$.  We proceed in reverse order, which roughly corresponds to the
  difficulty of the implications.
 
  {\em Implication (\ref{eq:moreci5})}: Let $\mathcal{L}=\{ 12|3 ,
  23|4 , 34|1 , 14|2 \}$.  We need to show that the vanishing ideal
  $\mathcal{I}(V_\cor(\mathcal{L}))$ contains $r_{12}$.  A primary
  decomposition of the CI ideal $I_{\mathcal{L}}$, which is radical,
  is given by $I_{\mathcal{L}}=\cap_{i=1}^3Q_i$ with the three
  components: 
  \begin{align*}
    Q_1 & =  \langle r_{24}r_{13}-1, \ldots \rangle,&& &
    Q_2 & =  \langle r_{24}r_{13}+1, \ldots \rangle,&& &
    Q_3 & =  \langle r_{12},r_{14},r_{23},r_{34} \rangle .
  \end{align*}
  The claim follows as only the variety of $Q_3$ intersects
  $\PD_{4,1}$; recall Example~\ref{ex:seth}.
  
  {\em Implication (\ref{eq:moreci4})}: For the relation $\mathcal{L}=\{
  12|3 , 13|4 , 14|2 \}$, the ideal $I_{\mathcal{L}}$ is radical and has
  a primary decomposition with the two components:
  \begin{align*}
    Q_1 & =  \langle r_{23}r_{24}r_{34}-1, \ldots \rangle, &
    Q_2 & =  \langle r_{12},r_{13},r_{14} \rangle .
  \end{align*}
  Since $r_{24}r_{34}r_{23}-1 < 0$ over $\PD_{4,1}$, the variety
  $V_{\CCC}(Q_1)$ does not intersect $\PD_{4,1}$. Therefore, every matrix
  $R=(r_{ij})$ in $V_\cor(\mathcal{L})$ has $r_{12}=0$.
   
  {\em Implication (\ref{eq:moreci3})}: For the relation $\mathcal{L}=\{ 12
  , 24|3 , 34|1 , 13|24 \}$, the ideal $I_{\mathcal{L}}$ is radical and has
  a primary decomposition with the two components:
  \begin{eqnarray*}
    Q_1 & = & \langle r_{14}^2r_{23}^2-r_{14}^2-r_{24}^2+1,
    r_{23}r_{34}-r_{24},r_{13}r_{14}-r_{34},
    \ldots \rangle , \\ 
    Q_2 & = & \langle r_{12},r_{13},r_{24},r_{34} \rangle.
  \end{eqnarray*}
  Since the polynomial
  \[
  r_{14}^2r_{23}^2-r_{14}^2-r_{23}^2r_{13}^2r_{14}^2+1 = 
  (1-r_{14}^2) + r_{14}^2r_{23}^2(1-r_{13}^2) 
  \]
  is in $Q_1$ but positive on $\PD_{4,1}$, the variety $V_\CCC(Q_1)$ does
  not intersect $\PD_{4,1}$.
  
  {\em Implication (\ref{eq:moreci2})}: For the relation $\mathcal{L}=\{ 12
  , 34|1 , 34|2 , 12|34 \}$, the ideal $I_{\mathcal{L}}$ is radical and has
  a primary decomposition with the components:
  \begin{eqnarray*}
    Q_1 & = & \langle r_{13}^2r_{24}^2-r_{13}^2-r_{24}^2+r_{34}^2,
    r_{23}r_{24}-r_{34}, \ldots 
    \rangle , \\
    Q_2 & = & \langle r_{12},r_{14},r_{23},r_{34} \rangle ,\\
    Q_3 & = & \langle r_{12},r_{13},r_{24},r_{34} \rangle.
  \end{eqnarray*}
  Only $Q_1$ does not already contain $r_{34}$.  Let $R=(r_{ij})$ be a positive
  definite matrix in $V_\CCC(Q_1)$. Since
  \[
  r_{13}^2r_{24}^2-r_{13}^2-r_{24}^2+r_{23}^2r_{24}^2 =
  r_{13}^2(r_{24}^2-1)+r_{24}^2(r_{23}^2-1)\in Q_1,
  \]
  the matrix entries satisfy $r_{13}=r_{24}=0$ and, thus,
  $r_{34}=r_{23}r_{24}=0$.
  
  {\em Implication (\ref{eq:moreci1})}: If $\mathcal{L}=\{ 12 , 34 , 13|24
  , 24|13 \}$, then $I_{\mathcal{L}}$ is radical and has a primary
  decomposition with the three components:
  \begin{eqnarray*}
    Q_1 & = & \langle
    r_{13}^2-r_{14}r_{23}-1,r_{24}-r_{13},r_{12},r_{34}\rangle, \\ 
    Q_2 & = & \langle r_{13}^2+r_{14}r_{23}-1,r_{24}+r_{13},r_{12},r_{34}
    \rangle, \\ 
    Q_3 & = & \langle r_{12},r_{13},r_{24},r_{34} \rangle.
  \end{eqnarray*}
  The varieties of $Q_1$ and $Q_2$ do not intersect $\PD_{4,1}$, which
  implies $r_{13}=0$ for the matrices in $V_\cor(\mathcal{L})$.  To
  see this, note that for a symmetric matrix $R=(r_{ij})$ with ones on
  the diagonal, it holds that $\det(R_{123,123})+(r_{14}+r_{23})^2\in
  Q_1$ and $\det(R_{123,123})+(r_{14}-r_{23})^2\in Q_2$.  Hence, if
  $R=(r_{ij})$ is a real matrix in $V_\CCC(Q_1)$ or $V_\CCC(Q_2)$ then
  it is not positive definite as $ \det(R_{123,123})=-(r_{14}\pm
  r_{23})^2\le 0$.
\end{proof}

\section{Singular loci of representable models}
\label{sec:singularity}

We now return to the problem of Question~\ref{q:smooth} for $m=4$,
that is, identify the relations $\mathcal{L}$ on the index set
$[m]=[4]$ for which the set $V_\pd(\mathcal{L})$ is a smooth manifold.
According to Theorem~\ref{thm:repr-decomp-variety} every conditional
independence model is a union of representable models.  Moreover, by
Lemma~\ref{lem:R-smooth}, we may equivalently consider the set of
correlation matrices $V_\cor(\mathcal{L})$.  The focus of this section
is thus the geometry of $V_\cor(\mathcal{L})$ when $\mathcal{L}$ is a
representable relation on $[m]=[4]$.

\subsection{Irreducible decomposition}
\label{subsec:repr-irred-decomp}

The set $V_\cor(\mathcal{L})$ associated with a representable relation
$\mathcal{L}$ cannot be further decomposed when only considering sets
defined by CI constraints.  However, there is no reason why
$V_\cor(\mathcal{L})$ should not further decompose in an irreducible
decomposition; recall (\ref{eq:irred-decomp}).  Indeed, computing
primary decompositions in {\tt Singular} we observe the following (We
note that $I_\mathcal{L} = I_\mathcal{L} : D^\infty$ for all
representable relations $\mathcal{L}$ on $[m]=[4]$):

\begin{proposition}
  \label{prop:irred-decomp-repr}
  If $\mathcal{L}$ is a representable relation on $[m]=[4]$, then the
  conditional independence ideal $I_\mathcal{L}$ is a prime ideal
  except when $\mathcal{L}$ is equivalent to one of the relations
  $\mathcal{L}_{15}$, $\mathcal{L}_{24}$, $\mathcal{L}_{28}$ and
  $\mathcal{L}_{37}$ listed in Table~\ref{tab:rep} in
  Appendix~\ref{sec:appendix}.
\end{proposition}

We now describe the primary decompositions of the four exceptional
representable relations.

\begin{example}
  \label{ex:pd-15}
  For the representable relation $\mathcal{L}_{15}=\{14,14|23,23,23|14\}$,
  the ideal $I_{\mathcal{L}_{15}}$ has 4 prime components:
  \begin{align*}
    Q_1 & =  \langle r_{12},r_{14},r_{23},r_{34} \rangle , &
    Q_2 & =  \langle r_{13},r_{14},r_{23},r_{24} \rangle ,\\
    Q_3 & =  \langle r_{14},r_{23},r_{12}+r_{34},r_{13}-r_{24} \rangle, &
    Q_4 & =  \langle r_{14},r_{23},r_{12}-r_{34},r_{13}+r_{24} \rangle.
  \end{align*}
  Hence, the model $V_\cor(\mathcal{L}_{15})$ is the union of four
  two-dimensional linear spaces intersected with the set of correlation
  matrices $\PD_{4,1}$.  Only matrices $R$ in $V_\cor(Q_3)$ and
  $V_\cor(Q_4)$ can represent $\mathcal{L}_{15}$ in the sense of
  $\mathcal{L}(\{ R\})=\mathcal{L}_{15}$.
\end{example}

\begin{example}
  \label{ex:pd-37}
  For $\mathcal{L}_{37}=\{12|3,12|4,34|1,34|2\}$, the ideal
  $I_{\mathcal{L}_{37}}$ has 4 two-dimensional prime components:
  \begin{align*}
    Q_1 & =  \langle r_{12},r_{13},r_{24},r_{34} \rangle , &
    Q_2 & =  \langle
    r_{12}-r_{34},r_{13}-r_{24},r_{14}-r_{23},r_{23}r_{24}-r_{34}
    \rangle ,\\ 
    Q_3 & =  \langle r_{12},r_{14},r_{23},r_{34} \rangle , &    
    Q_4 & =  \langle
    r_{12}+r_{34},r_{13}+r_{24},r_{14}+r_{23},r_{23}r_{24}-r_{34}
    \rangle .
  \end{align*}
  As in Example~\ref{ex:pd-15}, only two components, namely, $V_\cor(Q_2)$
  and $V_\cor(Q_4)$, contain matrices that represent $\mathcal{L}_{37}$.
  The points of $V_\cor(Q_2)$ and $V_\cor(Q_4)$ have the form $(x,x)$ and
  $(x,-x)$, respectively, where $x$ is on the conditional independence
  surface depicted in Figure~\ref{fig:comps24-28}(a).
\end{example}

\begin{example}
  \label{ex:pd-24}
  For $\mathcal{L}_{24}=\{12,23|14,24|3\}$, the ideal
  $I_{\mathcal{L}_{24}}$ has two prime components:
  \begin{align*}
    Q_1 & =  \langle
    r_{12},r_{23}r_{34}-r_{24},r_{13}r_{14}r_{34}-r_{14}^2-r_{34}^2+1\rangle,
    &  
    Q_2 & =  \langle r_{12},r_{23},r_{24} \rangle .
  \end{align*}
  Both of the 3-dimensional components intersect the set of correlation
  matrices $\PD_{4,1}$, and they intersect each other.  Only $V_\cor(Q_1)$
  contains representing matrices.  Note that $V_\cor(Q_1)$ is the image of
  the surface in $(r_{13},r_{14},r_{34})$-space given by
  $r_{13}r_{14}r_{34}-r_{14}^2-r_{34}^2+1=0$ under the 
  transformation setting $r_{24}=r_{23}r_{34}$ and leaving all other
  coordinates fixed.  Figure~\ref{fig:comps24-28}(b) displays this surface.
\end{example}

\begin{example}
  \label{ex:pd-28}
  For $\mathcal{L}_{28}=\{13|2,14,23|14,24|3\}$, the ideal
  $I_{\mathcal{L}_{28}}$ has two 2-dimensional prime components:
  \begin{align*}
    Q_1 & =  \langle r_{14},\,
    r_{12}r_{23}-r_{13},\,
    r_{23}r_{34}-r_{24},\,
    r_{12}^2+r_{34}^2-1
    \rangle,
    &
    Q_2 & =  \langle r_{13},r_{14},r_{23},r_{24} \rangle .
  \end{align*}
  The components intersect $\PD_{4,1}$ and each other.  The representing
  set $V_\cor(Q_1)$ is the image of a cylinder in
  $(r_{12},r_{23},r_{34})$-space under the transformation setting
  $r_{13}=r_{12}r_{23}$ and $r_{24}=r_{23}r_{34}$ and leaving the other
  coordinates fixed.
\end{example}

\begin{figure}[t]
  \centering
  (a)
  \includegraphics[width=5cm]{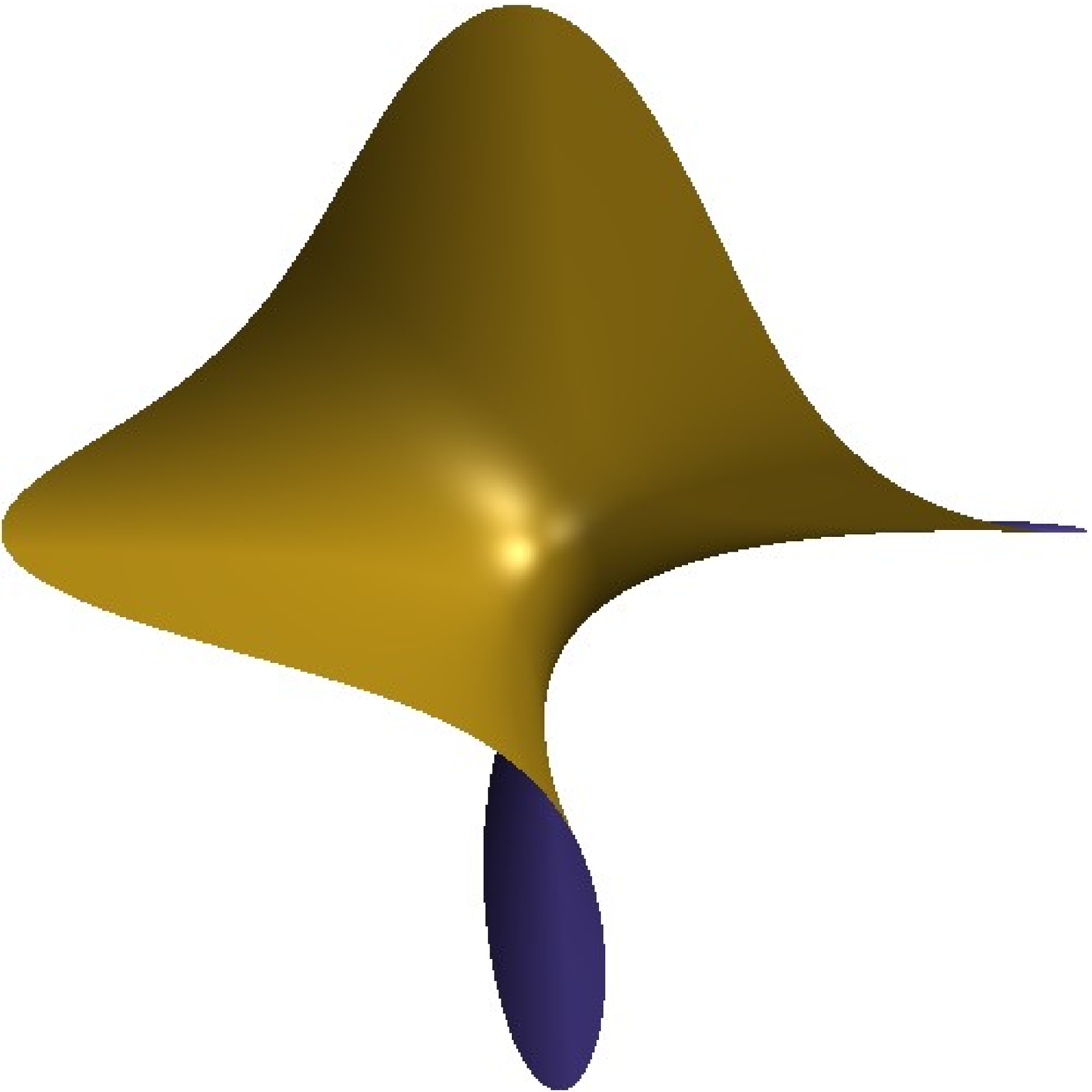}
  \hfill
  (b)
  \includegraphics[width=5cm]{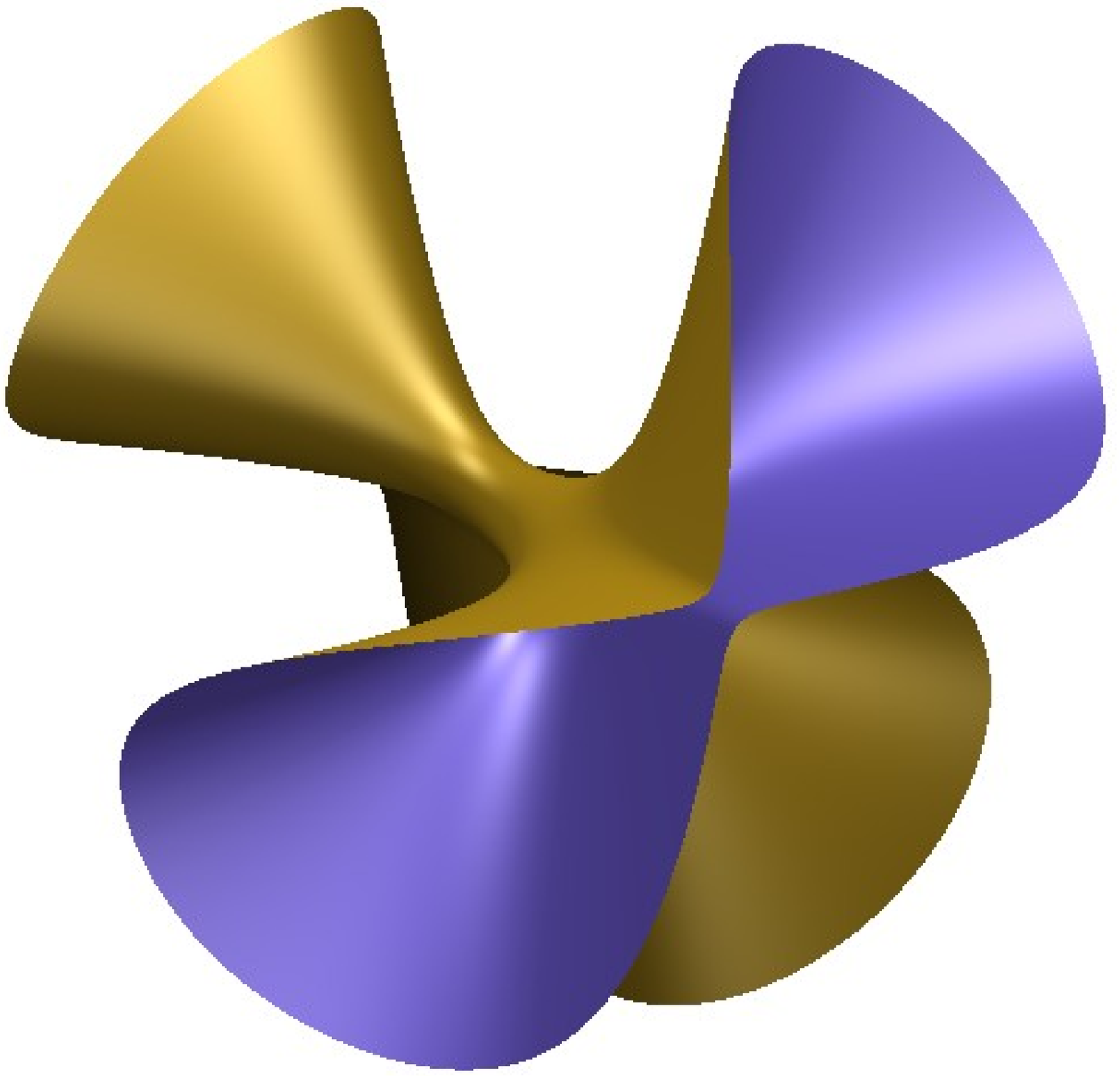}
  \caption{(a) Surface given by  $1\ind 2|3$, that is,
    $r_{12}r_{13}-r_{23}=0$. (b) Surface defined by
    $r_{13}r_{14}r_{34}-r_{14}^2-r_{34}^2+1=0$.  It arises for a component of
    $V_\cor(\mathcal{L}_{24})=V_\cor(\{12,23|14,24|3\})$.}
  \label{fig:comps24-28}
\end{figure}

\subsection{Singular points}
\label{subsec:singpoints}

Suppose $V$ is an algebraic variety in the space
$\mathbb{S}_{m,1}(\CCC)$ of complex symmetric $m\times m$ matrices
with ones on the diagonal.  Let $\mathcal{I}(V)$ be the ideal of
polynomials vanishing on $V$.  Choose $\{f_1,f_2,\ldots, f_{\ell}\}
\subset \RRR[\mathbf{r}]$ to be a finite generating set of
$\mathcal{I}(V)$, and define $J(\mathbf{r})$ to be the $\ell \times
\binom{m}{2}$ Jacobian matrix with $(k,ij)$ entry equal to $\partial
f_k(\mathbf{r})/\partial r_{ij}$.  It can be shown that the maximum
rank the Jacobian matrix achieves over $V$ is equal to
$\codim(V)=\binom{m}{2}-\dim(V)$ and, in particular, independent of
the choice of the generating set $\{f_1,f_2,\ldots, f_{\ell}\}$.  See
for instance \cite[\S3]{benedetti:1990} for a proof of this fact as we
as Lemma~\ref{lem:reg-points}, below.

\begin{definition}
  \label{def:singular-point}
  If the variety $V \subseteq \mathbb{S}_{m,1}(\CCC)$ is irreducible
  then a matrix $R=(r_{ij})\in V$ is a {\em singular point} if the
  rank of $J(R)$ is smaller than $\codim(V)$.  If $V$ is not
  irreducible, then the singular points are the singular points of the
  irreducible components of $V$ together with the points in the
  intersection of any two  irreducible components.
\end{definition}

When presented with a set of correlation matrices
$V_\cor(\mathcal{L})$ arising from a CI relation $\mathcal{L}$, it is
useful to study the singularities of the variety
$V_\CCC(I_\mathcal{L})$.

\begin{lemma}
  \label{lem:reg-points}
  The set of all points in $V_\cor(\mathcal{L})$ that are non-singular
  points of $V_\CCC(I_\mathcal{L})$ is a smooth manifold.
\end{lemma}

A computational approach to the smoothness problem is thus to
calculate the locus of singular points of $V_\CCC(I_\mathcal{L})$,
using for instance the available routines in {\tt Singular}.  To
determine irrelevant components that do not intersect the set of
correlation matrices $\PD_{m,1}$, we saturate the ideal
$S_\mathcal{L}$ describing this singular locus on the product of
principal minors $D$ and then compute a primary decomposition of
$(S_\mathcal{L}:D^\infty)$.  If the singular locus is seen not to
intersect $\PD_{m,1}$ then the computation proves that
$V_\cor(\mathcal{L})$ is a smooth manifold.  If, however, there are
correlation matrices that are singular points of
$V_\CCC(I_\mathcal{L})$, then we may not yet conclude that
$V_\cor(\mathcal{L})$ is non-smooth around these points.  An algebraic
obstacle is the fact that $I_\mathcal{L}$ might differ from the
vanishing ideal $\mathcal{I}(V_\cor(\mathcal{L}))$.  However, even if
$I_\mathcal{L}=\mathcal{I}(V_\cor(\mathcal{L})$, then algebraic
singularity of a point as specified in
Definition~\ref{def:singular-point} need not imply
that the positive definite set $V_\cor(\mathcal{L})$ fails to be a
smooth manifold in a neighborhood of this point.  For a classical example
of a real algebraic curve with this feature; see \cite[Example
3.3.12(b)]{bochnak:1998}.

On the three-element set $[m]=[3]$, and up to equivalence,
$\mathcal{L}=\{12,12|3\}$ is the only relation for which
$V_\cor(\mathcal{L})$ is not a smooth manifold.  The following
proposition explains the drop in rank of the Jacobian in a generalized
scenario.

\begin{proposition}
  \label{thm:ijcd}
  Let $f_1=\det(R_{iC_1,jC_1}),f_2=\det(R_{iC_2,jC_2})\in\RRR[\mathbf{r}]$
  be the two determinants encoding the relation
  $\mathcal{L}=\{ij|C_1,ij|C_2\}$ on $[m]$.  Let $J(R)$ be the $2\times
  \binom{m}{2}$ Jacobian matrix for $f_1,f_2$ evaluated at a correlation
  matrix $R$.  Then the maximal rank of $J(R)$ over $V_\cor(\mathcal L)$ is
  two but this rank drops to one exactly when $R$ satisfies the two
  conditional independence constraints
  \begin{equation}
    \label{eq:ijcd}
    i \ind j(C_1\bigtriangleup C_2)|(C_1\cap C_2) \quad\text{and}\quad
    j \ind i(C_1\bigtriangleup C_2)|(C_1\cap C_2).
  \end{equation}
  Here, $C_1\bigtriangleup C_2=(C_1\setminus C_2)\cup (C_2\setminus C_1)$
  is the symmetric difference.
\end{proposition}


\begin{proof}
  Let $F=C_1 \cap C_2$, $C=C_1\setminus C_2$ and $D=C_2\setminus C_1$.
  Then $C$, $D$ and $F$ are pairwise disjoint, and $f_1=\det(R_{iCF,jCF})$
  and $f_2=\det(R_{iDF,jDF})$.

  Before turning to the study of the Jacobian $J(R)$, we note that, by
  Proposition \ref{prop:cic}, the condition (\ref{eq:ijcd}) is equivalent
  to the vanishing of five Schur complements:
  \begin{align}
    \label{eq:ijcd1}
    r_{ij} - R_{i,F}R_{F,F}^{-1}R_{F,j} & = 0, \\ 
    \label{eq:ijcd2}
    R_{C,i} - R_{C,F}R_{F,F}^{-1}R_{F,i} & =  0, &
    R_{C,j} - R_{C,F}R_{F,F}^{-1}R_{F,j} & =  0, \\ 
    \label{eq:ijcd4}
    R_{D,i} - R_{D,F}R_{F,F}^{-1}R_{F,i}& =  0, &
    R_{D,j} - R_{D,F}R_{F,F}^{-1}R_{F,j}& = 0.
  \end{align}
  Below we sometimes use the following shorthand for such
  Schur complements:
  \[
  R_{A,B|F} := R_{A,B} - R_{A,F}R_{F,F}^{-1}R_{F,B}.
  \]

  Depending on whether or not $r_{kl}$ is a `symmetric' entry of the matrix
  defining $f_h$, the partial derivative $\partial f_h/\partial r_{kl}$ is
  equal to the $(k,l)$ cofactor or the sum of the $(k,l)$ and $(l,k)$
  cofactors.  When discussing these derivatives we always suppress
  the signs that appear when calculating a cofactor or switching two
  columns in a determinant.  It is easy to check that these signs do not
  affect the proof.  When writing out cofactors we use the notation
  $|\cdot|=\det(\cdot)$.
  
  The column of $J(R)$ associated with $r_{ij}$ contains two non-zero
  entries because
  \begin{align}
    \label{eq:ijCD-col1}
    \frac{\partial f_1}{\partial r_{ij}} &= |R_{CF,CF}|, &
    \frac{\partial f_2}{\partial r_{ij}} &= |R_{DF,DF}|, &
  \end{align}
  are two principal minors of $R$.  Hence, $J(R)$ has rank either one
  or two.

  
  {\em Necessity of (\ref{eq:ijcd}):} The correlations $r_{ic}$ with $c\in
  C$ do not appear in $f_2$.  Therefore, for the rank of $J(R)$ to drop to
  one, it is necessary that $\partial f_1/\partial r_{ic}=0$ for all $c \in
  C$.  This derivative is equal to
  \begin{equation}
    \label{eq:f1-ic}
    \frac{\partial f_1}{\partial r_{ic}}  = 
     \begin{vmatrix}
        R_{C,j} & R_{C,C\setminus c} & R_{C,F} \\
        R_{F,j} & R_{F,C\setminus c} & R_{FF}
      \end{vmatrix}  =  \left|R_{FF}\right|\left| R_{C,jC\setminus c} -
      R_{C,F}R_{F,F}^{-1}R_{F,jC\setminus c}\right|.
  \end{equation}
  (Note that, due to our convention of not distinguishing indices and
  singleton index set, $jC\setminus c = (C\cup\{j\})\setminus\{c\}$.)  The
  matrix $R_{C,jC\setminus c} - R_{C,F}R_{F,F}^{-1}R_{F,jC\setminus c}$ is
  obtained by replacing the $c$-th column of $R_{C,C|F}$
  by $R_{C,j} - R_{C,F}R_{F,F}^{-1}R_{F,j}$. Since 
  $R_{C,C|F}$ is positive definite, and the last
  determinant in (\ref{eq:f1-ic}) is zero for all $c\in C$, it follows that
  $R_{C,j} - R_{C,F}R_{F,F}^{-1}R_{F,j}=0$.  In other words,
  the second equation in (\ref{eq:ijcd2}) holds.   
  
  Similarly, the rank of $J(R)$ can only be one if $\partial
  f_1/\partial r_{jc}=0$ for all $c\in C$.  This implies the first
  equation in (\ref{eq:ijcd2}).  Treating $f_2$ analogously,
  (\ref{eq:ijcd4}) 
  also needs to hold.
  
  The remaining condition, (\ref{eq:ijcd1}), is a consequence of the matrix
  $R$ being in $V_\cor(\mathcal{L})$.  In the current amended notation, the
  first defining CI couple is $ij|CF$.  By iterated conditioning (iterated
  Schur complements), this conditional independence holds if and only if
  the determinant of the conditional covariance matrix
   \begin{equation}
    \label{eq:covgivenF}
    R_{iC,jC|F}   =   \begin{bmatrix}
        R_{i,j|F} & R_{i,C|F} \\
        R_{C,j|F} & R_{C,C|F}
    \end{bmatrix} =
    \begin{bmatrix}
        R_{i,j|F} & 0 \\
        0 & R_{C,C|F}
    \end{bmatrix}
  \end{equation}
  is zero.  It follows that
  $R_{i,j|F}=r_{ij}-R_{i,F}R_{F,F}^{-1}R_{F,j}=0$, which is
  (\ref{eq:ijcd1}).
  
  {\em Sufficiency of (\ref{eq:ijcd}):} If
  (\ref{eq:ijcd1})-(\ref{eq:ijcd4}) hold, many partial derivatives are
  zero.  First, (\ref{eq:f1-ic}) implies $\partial f_1/\partial
  r_{ic}=\partial f_1/\partial r_{jc}=\partial f_2/\partial
  r_{id}=\partial f_2/\partial r_{jd}=0$ for $c\in C$ and
  $d\in D$.
  
  Second, consider two distinct indices $c,c'\in C$.  The derivative
  $\partial f_1 /\partial r_{cc'}$ is the sum of two cofactors. The first
  cofactor is
  \begin{align*}
    \begin{vmatrix}
      R_{iC\setminus c, jC\setminus c'} &
      R_{iC\setminus c, F} \\
      R_{F,jC\setminus c'} & R_{F,F}
    \end{vmatrix} = \left|R_{F,F}\right|
    \begin{vmatrix}
      R_{i,j|F} & R_{i,C \setminus c'|F}  \\
      R_{C\setminus c,j|F} & R_{C\setminus c, C\setminus c'|F}
    \end{vmatrix}    =0,
  \end{align*}
  because, by (\ref{eq:ijcd1}) and (\ref{eq:ijcd2}), the
  last determinant is that of a matrix with first row and column zero.
  The other cofactor is obtained by switching $c$ and $c'$ and also zero.
  Hence, $\partial f_1/\partial r_{cc'}=0$.  Similarly, $\partial
  f_2/\partial r_{dd'}=0$ for two distinct indices $d,d'\in D$.
 
  Third, if $c\in C$ and $f\in F$, then $\partial f_1/\partial r_{cf}$
  is the sum of two cofactors.  Using (\ref{eq:ijcd1}) and
  (\ref{eq:ijcd2}), one cofactor is seen to be
      \begin{align*}
        \begin{vmatrix}
          R_{i,C} & R_{i,F\setminus f} & r_{ij} \\
          R_{C\setminus c, C} & R_{C\setminus c, F\setminus f} &
          R_{C\setminus c,j} \\ 
          R_{F,C} & R_{F,F\setminus f} & R_{F,j}
          \end{vmatrix}
        =
        \begin{vmatrix}
            R_{i,F}R_{F,F}^{-1}R_{F,C} & R_{i,F\setminus f} &
            R_{i,F}R_{F,F}^{-1}R_{F,j}  \\ 
            R_{C\setminus c, C} & R_{C\setminus c, F\setminus f} &
            R_{C\setminus c,F}R_{F,F}^{-1}R_{F,j}  \\
            R_{F,C} & R_{F,F\setminus f} & R_{F,F}R_{F,F}^{-1}R_{F,j}
          \end{vmatrix} .
      \end{align*}
      Let $R_{F,F}^{-1}R_{F,j}(f)$ be the $f$-th entry of the vector
      $R_{F,F}^{-1}R_{F,j}$.  The above cofactor is
      \begin{align*}
        R_{F,F}^{-1}R_{F,j}(f) \,\cdot\, \begin{vmatrix}
          R_{i,F}R_{F,F}^{-1}R_{F,C} & R_{i,F}R_{F,F}^{-1}R_{F,F}  \\
          R_{C\setminus c, C} & R_{C\setminus c, F}\\
          R_{F,C} & R_{F,F}
          \end{vmatrix}  = 0
      \end{align*}
      because $(R_{i,F}R_{F,F}^{-1}R_{F,C},\,R_{i,F}R_{F,F}^{-1}R_{F,F})$
      is a linear combination of rows of the matrix $(R_{F,C},\, R_{F,F})$.
      Similarly, the other cofactor is zero and, thus, $\partial
      f_1/\partial r_{cf} = 0$.  The vanishing of $\partial f_2/\partial
      r_{df}$ for $d\in D$ is analogous.
  
      Our calculations show that only the columns of $J(R)$ associated with
      $r_{ij}$, $r_{if}$, $r_{jf}$ and $r_{ff'}$ for $f\neq f'\in F$ may be
      non-zero.  To establish that $J(R)$ has rank one we show that these
      columns are all multiples of the one for $r_{ij}$ given in
      (\ref{eq:ijCD-col1}).
      
      Using the second equation in (\ref{eq:ijcd2}), we have that
      \begin{align*}
        \frac{\partial f_1}{\partial r_{if}} & =  \begin{vmatrix}
            R_{C,C} & R_{C,F\setminus f} & R_{C,j} \\
            R_{F,C} & R_{F,F\setminus f} & R_{Fj}
          \end{vmatrix} =
          \begin{vmatrix}
              R_{C,C} & R_{C,F\setminus f} & R_{C,F}R_{F,F}^{-1}R_{F,j} \\
              R_{F,C} & R_{F,F\setminus f} & R_{F,F}R_{F,F}^{-1}R_{Fj}
          \end{vmatrix}.
      \end{align*}
      Therefore, we obtain that 
      \begin{align*}
        \frac{\partial f_1}{\partial r_{if}} &= 
        R_{F,F}^{-1}R_{F,j}(f)\,\cdot\, \begin{vmatrix} 
          R_{C,C} & R_{C,F} \\
          R_{F,C} & R_{F,F}
        \end{vmatrix} =
        R_{F,F}^{-1}R_{F,j}(f)\, \frac{\partial
          f_1}{\partial r_{ij}}.
      \end{align*}
      The derivatives $\partial f_1/\partial r_{jf}$, $\partial
      f_2/\partial r_{if}$ and $\partial f_2/\partial r_{jf}$ are
      similar multiples of the corresponding derivatives with respect
      to $r_{ij}$.
      
      The two remaining cases $\partial f_1/\partial r_{ff'}$ and $\partial
      f_2/\partial r_{ff'}$ are again analogous, and we only consider the
      former.  This derivative is the sum of two cofactors, one being
      \begin{align*}
        \begin{vmatrix}
          R_{i,C} & R_{i,F\setminus f'} & r_{ij} \\
          R_{C, C} & R_{C, F\setminus f'} & R_{C,j} \\
          R_{F\setminus f,C} & R_{F\setminus f, F\setminus f'} &
          R_{F\setminus f,j}
        \end{vmatrix} =
        \begin{vmatrix}
          R_{i,C} & R_{i,F\setminus f'}  & R_{i,F}R_{F,F}^{-1}R_{F,j}  \\
          R_{C, C} & R_{C, F\setminus f'} & R_{C,F}R_{F,F}^{-1}R_{F,j}  \\
          R_{F\setminus f,C} & R_{F\setminus f, F\setminus f'} &
          R_{F\setminus f,F}R_{F,F}^{-1}R_{F,j}
        \end{vmatrix} 
      \end{align*}
      where (\ref{eq:ijcd1}) and the second equation in
      (\ref{eq:ijcd2}) were applied.  Using the first equation in
      (\ref{eq:ijcd2}), the cofactor is seen to be equal to
      \begin{align*}
        R_{F,F}^{-1}R_{F,j}(f') & \,\cdot\, 
        \begin{vmatrix}
          R_{i,F}R_{F,F}^{-1}R_{F,C} & R_{F,F}R_{F,F}^{-1}R_{i,F} \\
          R_{C, C} & R_{C, F} \\
          R_{F\setminus f,C} & R_{F\setminus f, F}
        \end{vmatrix}
        \\
        &=
        R_{F,F}^{-1}R_{F,j}(f')\,\cdot\,
        R_{F,F}^{-1}R_{F,i}(f)\,\cdot\,
        \begin{vmatrix} 
            R_{C, C} & R_{C, F} \\
            R_{F,C} & R_{F, F}
          \end{vmatrix} 
      \end{align*}
      The other cofactor is obtained by switching $f$ and $f'$ and, thus,
      \begin{align*}
        \frac{\partial f_1}{\partial r_{ff'}} & =
        \left[R_{F,F}^{-1}R_{F,j}(f')\,\cdot\,
          R_{F,F}^{-1}R_{F,i}(f) + 
          R_{F,F}^{-1}R_{F,j}(f)\,\cdot\,
          R_{F,F}^{-1}R_{F,i}(f')\right]\,
        \frac{\partial f_1}{\partial r_{ij}}.
      \end{align*}
      We have thus proven that the rank of $J(R)$ is 
      one when (\ref{eq:ijcd}) holds.
\end{proof}

\subsection{Singular loci of representable models on four variables}
\label{subsec:sing-loci}

Implementing the computational approach described in
Section~\ref{subsec:singpoints}, we find the following result for $m=4$ random
variables.  Note that Proposition~\ref{thm:ijcd} applies to the
representable relations with index 29 and 32.

\begin{theorem}
  \label{thm:sing-repr}
  If $\mathcal{L}$ is a representable relation on $[m]=[4]$, then
  $V_\cor(\mathcal{L})$ is a smooth manifold unless $\mathcal{L}$ is
  equivalent to one of 12 relations $\mathcal{L}_i$ with index
  $i\in\{14, 15, 20, 24,28,29,30, 32, 36,37, 46, 51 \}$ listed in
  Table~\ref{tab:rep}.
\end{theorem}
\begin{proof}
  Going through 53 possible cases, the computation identifies 41 models as
  smooth according to Lemma~\ref{lem:reg-points}.  The remaining 12 models
  are algebraically singular.  Our analysis of tangent cones below shows
  that these 12 models are indeed not smooth manifolds (compare
  Theorem~\ref{thm:tcone-repr}).
\end{proof}

We now give some more details on the singularities of the 12 relations
listed in Theorem~\ref{thm:sing-repr}.  They can be grouped into 3
categories:

\bigskip {\it (a) Union of smooth components:} If $i\in\{24,28\}$ then
$V_\cor(\mathcal{L}_i)$ is the union of two components that are both
smooth manifolds; compare Examples~\ref{ex:pd-24} and \ref{ex:pd-28}.
In each case the singular locus is simply the intersection of the two
components, which gives the surface defined by $\langle
r_{12},r_{23},r_{24},r_{13}r_{14}r_{34}-r_{14}^2-r_{34}^2+1 \rangle$
and the circle defined by $
  \langle r_{13},r_{14},r_{23},r_{24},r_{12}^2+r_{34}^2-1 \rangle .$
  
  If $i\in\{15,37\}$ then $V_\cor(\mathcal{L}_i)$ is the union of four
  smooth components; compare Examples~\ref{ex:pd-15} and
  \ref{ex:pd-37}.  The singular locus is again obtained by forming
  intersections of components.  In each case the singular locus has 4
  components that for $i=15$ are given by
  \begin{align*}
    &\langle r_{12},r_{14},r_{23},r_{34},r_{13}-r_{24} \rangle, 
    &\langle r_{12},r_{14},r_{23},r_{34},r_{13}+r_{24} \rangle,\\
    &\langle r_{13},r_{14},r_{23},r_{24},r_{12}-r_{34} \rangle,
    &\langle r_{13},r_{14},r_{23},r_{24},r_{12}+r_{34} \rangle, 
  \end{align*}
  and for $i=37$ by
  \begin{align*}
    &\langle r_{12},r_{13},r_{24},r_{34},r_{14}-r_{23} \rangle,
    &\langle r_{12},r_{13},r_{24},r_{34},r_{14}+r_{23} \rangle, \\
    &\langle r_{12},r_{14},r_{23},r_{34},r_{13}-r_{24} \rangle,
    &\langle r_{12},r_{14},r_{23},r_{34},r_{13}+r_{24} \rangle.
  \end{align*}
  
  {\it (b) Singular at identity matrix:} The six models with $i\in\{14, 20,
  30, 36, 46, 51\}$, have the identity matrix as their only singular
  point.
  
  {\it (c) Singular at almost diagonal matrices:} Two cases remain.
  If $i=29$, the correlation matrices that are singularities have the
  entries other than $r_{14}$ equal to zero.  For $i=32$, the
  singularities have the entries other than $r_{34}$ equal to zero.

\bigskip

Since algebraic singularity need not imply failure of smoothness, we now
study the local geometry of the sets $V_\cor(\mathcal{L})$ at their
algebraic singularities.  This local geometry is represented by the tangent
cone, which is also related to asymptotic distribution theory for
statistical tests \cite{drton:lrt}.

\begin{definition}
  A {\em tangent direction\/} of $V_\cor(\mathcal{L})$ at the correlation
  matrix $R_0\in \PD_{m,1}$ is a matrix in $\RRR^{m\times m}$ that is the
  limit of a sequence $\alpha_n(R_n-R_0)$, where the $\alpha_n$ are
  positive reals and the $R_n \in V_\cor(\mathcal{L})$ converge to $R_0$.
  The {\em tangent cone\/} $\mathit{TC}_{\mathcal{L}}(R_0)$ is the closed
  cone made up of all these tangent directions.
\end{definition}

The representable relations $\mathcal{L}_i$ with $i\in\{15,24,28,37\}$
define unions of smooth manifolds.  Their singularities lie in the
intersection of two or more of the smooth components, and the tangent cone
is then simply the union of the tangent spaces of the smooth components
containing a considered singularity.

Our strategy to determine the tangent cones of the remaining 8 singular
representable models is again algebraic.  Let the correlation matrix
$R_0\in\PD_{m,1}$ correspond to a root of the polynomial
$f\in\RRR[\mathbf{r}]$.  Write
\begin{equation*}
  f(R) \, = \,\sum_{h=l}^L f_h(R - R_0 )
\end{equation*}
as a sum of homogeneous polynomials $f_h$ in $R-R_0$, where $f_h(t)$
has degree $h$ and $f_l\not=0$.  Since $f(R_0)=0$, the minimal degree
$l$ is at least one, and we define $f_{R_0,\min}=f_l$.  The {\em
  algebraic tangent cone\/} of $V_\cor(\mathcal{L})$ at $R_0$ is the
real algebraic variety defined by the {\em tangent cone ideal\/}
\begin{equation}
  \label{eq:ac-ideal}
  \left\{f_{R_0,\min} \;:\;
      f\in\mathcal{I}(V_\cor(\mathcal{L}))
    \right\}\subset\RRR[\mathbf{r}].
\end{equation}
The algebraic tangent cone contains the tangent cone
$\mathit{TC}_{\mathcal{L}}(R_0)$; see e.g.~\cite[\S2.3]{drton:2009}.  In
our setup we work with the ideal
$I_\mathcal{L}\subseteq\mathcal{I}(V_\cor(\mathcal{L}))$ and, thus, consider
the cone $\mathit{AC}_{\mathcal{L}}(R_0)$ given by the real algebraic
variety of the ideal
\begin{equation}
  \label{eq:acL-ideal}
  C_\mathcal{L}(R_0)=
  \left\{f_{R_0,\min} \;:\;
      f\in I_\mathcal{L}
    \right\}\subset\RRR[\mathbf{r}].
\end{equation}
The cone $\mathit{AC}_{\mathcal{L}}(R_0)$ always contains the algebraic
tangent cone. Therefore, we still have the inclusion
$\mathit{TC}_{\mathcal{L}}(R_0)\subseteq\mathit{AC}_{\mathcal{L}}(R_0)$.
The ideal $C_\mathcal{L}(R_0)$ in (\ref{eq:acL-ideal}) can be computed
using Gr\"obner basis methods that are implemented, for instance, in the
{\tt tangentcone} command in {\tt Singular}.

\begin{theorem}
  \label{thm:tcone-repr}
  If $\mathcal{L}_i$ is one of the 8 representable relations on $[m]=[4]$
  with index $i\in\{14, 20,29,30, 32, 36, 46, 51 \}$, then at all
  singularities $R_0$ of $V_\cor(\mathcal{L}_i)$ the tangent cone
  $\mathit{TC}_{\mathcal{L}}(R_0)$ is equal to the algebraically defined
  cone $\mathit{AC}_{\mathcal{L}}(R_0)$.  In particular, the models
  $V_\cor(\mathcal{L}_i)$ are indeed non-smooth.
\end{theorem}
\begin{proof}
  The six models with $i\in\{14, 20, 30, 36, 46, 51\}$, have the
  identity matrix $\mathit{Id}$ as their only singular correlation
  matrix.  The cone ideals are
  \begin{align}
    \label{eq:tangentcone1}
    C_{\mathcal{L}_{14}}(\mathit{Id}) = C_{\mathcal{L}_{46}}(\mathit{Id})
    &= \langle 
    r_{14}, r_{23},  r_{12}r_{24}+r_{13}r_{34} \rangle,  \\
    \label{eq:tangentcone2}
    C_{\mathcal{L}_{20}}(\mathit{Id}) = C_{\mathcal{L}_{51}}(\mathit{Id})
    &= \langle 
    r_{14}, r_{23}, r_{12}r_{24}-r_{13}r_{34} \rangle ,\\
    \label{eq:tangentcone3}
    C_{\mathcal{L}_{30}}(\mathit{Id}) &=
    \langle r_{14}, r_{23}, r_{12}r_{13}-r_{24}r_{34} \rangle,\\
    \label{eq:tangentcone4}
    C_{\mathcal{L}_{36}}(\mathit{Id}) &= \langle r_{12}, r_{34},
    r_{13}r_{23}-r_{14}r_{24} \rangle.
  \end{align}
  The latter three ideals are equivalent under permutation of the indices
  in $[m]=[4]$.
  
  For $\mathcal{L}_{29}$, the singular points $R_0=(\rho_{ij}^0)$ have all
  off-diagonal entries zero except for possibly $\rho_{14}^0$ which can be
  any number in $(-1,1)$.  The cone ideal varies continuously with
  $\rho_{14}^0$:
  \begin{align}
    \label{eq:tangentcone5}
    C_{\mathcal{L}_{29}}(R_0) 
    &= \langle \;r_{23},\,
    r_{13}(r_{12} -\rho_{14}^0r_{24})+
    r_{34}(r_{24}-\rho_{14}^0 r_{12}) \;\rangle.
  \end{align}
  The algebraic cones in this family can be transformed into each other
  by an invertible linear transformation.
  
  For $\mathcal{L}_{32}$, the singular points $R_0=(\rho_{ij}^0)$ have all
  off-diagonal entries zero except for possibly $\rho_{34}^0$ which can be
  any number in  $(-1,1)$.  The cone ideal, however, does not depend on the
  value of $\rho_{34}^0$:
  \begin{equation}
    \label{eq:tangentcone6}
    C_{\mathcal{L}_{32}}(R_0) = \langle r_{12}, r_{13} r_{23} - r_{14}
    r_{24} \rangle.
  \end{equation}
  
  In each case, it can be shown that all vectors in
  $\mathit{AC}_{\mathcal{L}_i}(R_0)$ are indeed tangent directions for
  $V_\cor(\mathcal{L}_i)$. We prove the result for $i=29$; the other 7
  cases are similar.  

  \smallskip
   {\em Tangent cone of $V_\cor(\mathcal{L}_{29})$:}
  The ideal 
  \begin{equation*}
    I_{\mathcal{L}_{29}} = \langle r_{23}, -r_{14}^2
  r_{23}+r_{13} r_{14} r_{24}+r_{12} r_{14} r_{34}-r_{12}
  r_{13}-r_{24} r_{34}+r_{23} \rangle.
  \end{equation*}
  Let $\boldsymbol{r}_0=(0,0,\rho,0,0,0)$ with $|\rho|<1$ be a
  singular point and $R_0$ the corresponding correlation matrix.  Both
  $TC_{\mathcal{L}_{29}}(R_0)$ and $AC_{\mathcal{L}_{29}}(R_0)$ are
  closed sets, and we may thus consider a generic direction
  $\boldsymbol{t}=(t_{12},t_{13},t_{14},t_{23},t_{24},t_{34})$ in the
  cone $AC_{\mathcal{L}_{29}}(R_0)$ given by the ideal
  $C_{\mathcal{L}_{29}}(R_0)$ in (\ref{eq:tangentcone5}).
  We may assume $\rho t_{12} - t_{24} \neq 0$, and obtain
  \begin{equation}
    \label{eq:tangdir29}
    \boldsymbol{t}=\left(t_{12},t_{13},t_{14},0,t_{24},\frac{t_{13}(t_{12}-\rho
    t_{24})}{\rho t_{12} - t_{24}}\right).
  \end{equation}
  Let 
  $$\boldsymbol{r}_n = \left(
    \frac{t_{12}}{n},\frac{t_{13}}{n},\rho+\frac{t_{14}}{n},0,\frac{t_{24}}{n},
    \frac{nt_{13}(t_{12}- \rho t_{24}) - t_{13}t_{14}t_{24}}{n^2(\rho
      t_{12}-t_{24})+nt_{12}t_{14}} \right).$$
  It is easy to show that
  $\boldsymbol{r}_n \in V_\cor(\mathcal{L}_{29})$ for large $n$; and
  $\boldsymbol{r}_n \rightarrow \boldsymbol{r}_0$ and
  $n(\boldsymbol{r}_n-\boldsymbol{r}_0) \rightarrow \boldsymbol{t}$ as
  $n \rightarrow \infty$. Thus, $\boldsymbol{t} \in
  TC_{\mathcal{L}_{29}}(R_0)$, and it follows that
  $TC_{\mathcal{L}_{29}}(R_0)=AC_{\mathcal{L}_{29}}(R_0)$.
\end{proof}

\section{Conclusion}
\label{sec:conclusion}

We conclude by pointing out some interesting features of our
computational results for $m=4$.  First, the model associated with a
representable relation need not correspond to an irreducible variety.
It can be a union of several distinct irreducible components that all
intersect the cone of positive definite matrices; see
Examples~\ref{ex:pd-24} and \ref{ex:pd-28} in which the components all
have the same dimension.

Second, Examples~\ref{ex:pd-24} and \ref{ex:pd-28} also provide a
negative answer to Question 7.11 in \cite{drton:2009}.  This question
asked whether Gaussian conditional independence models that are smooth
locally at the identity matrix are smooth manifolds.  The singular
loci of these examples, however, do not contain the identity matrix.
All other singular models are singular at the identity matrix, and in
fact, the identity matrix is often the only singularity (recall
Section~\ref{subsec:sing-loci}).

Our final comment is based on the observation that Gaussian conditional
independence models for $m=3$ variables are smooth except for the model
given by $ij$ and $ij|k$, and that singular models can arise more generally
when combining two CI couples $ij|C$ and $ij|D$ (recall
Proposition~\ref{thm:ijcd}).  This observation may lead one to guess that
if a complete relation $\mathcal{L}$ does not contain two CI couples $ij|C$
and $ij|D$ that repeat the pair $ij$, then the model $V_\pd(\mathcal{L})$
is smooth.  Unfortunately, this is false, again because of
Examples~\ref{ex:pd-24} and \ref{ex:pd-28}.

\begin{appendix}

\section{Lists of relations and CI implications}
\label{sec:appendix}

In this appendix we provide encyclopedic information about conditional
independence of four Gaussian random variables.  This comprehends the
listing of all representable and complete relations, as well as all
Gaussian CI implications.

\subsection{Representable and complete relations}

Counting up to equivalence, there are 53 representable relations on
four variables.  They are listed in Table~\ref{tab:rep} below, where
the symbol $\ast$ is used to denote all possible conditioning sets.
With this convention, the symbol $12|\ast$, for instance, expands to
$\{12,12|3,12|4,12|34\}$.  In the table we indicate whether the model
is singular and give the equivalent dual of each representable
relation.  We include a permutation of the indices that provides the
equivalence.  Using cycle notation, an empty entry stands for the
identity.

The remaining $101-53=48$ equivalence classes of complete but
not representable relations are listed in Table \ref{tab:com}.
Each cell in the second column provides, row-by-row, the representable
relations in the minimal representable decomposition from
Theorem \ref{thm:repr-decomp-variety}.  Their intersection gives the
considered complete relation.  For each representable relation in the
table cell, we also provide, in the third column, the label of its
equivalent representable relation in Table \ref{tab:rep} and the
permutation that transforms the equivalent relation to the current
one.

In the introduction we mentioned that graphical models are smooth CI
models.  Over 4 nodes, there are 11 unlabelled undirected graphs.
Figure~\ref{fig:semigaussoids} from Section~\ref{subsec:all-models}
shows for each of the 10 non-empty graphs the corresponding
representable relation from Table~\ref{tab:rep}.  Up to equivalence,
there are 10 additional graphical models associated with acyclic
digraphs on 4 nodes.  The 10 digraphs are shown in
Figure~\ref{fig:dags} together with the corresponding representable
relations.  We remark that the representable relations
$\mathcal{L}_{35}$, $\mathcal{L}_{40}$, $\mathcal{L}_{44}$,
$\mathcal{L}_{48}$ and $\mathcal{L}_{49}$ determine graphical models
based on mixed graphs with directed and bi-directed edges.  The
corresponding 5 graphs are shown in \cite[Fig.~10]{richardson:2003}.
Two further representable relations correspond to chain graphs:
$\mathcal{L}_{13}$ is given by the so-called LWF interpretation of the
graph in \cite[Fig.~1]{andersson:2001} and $\mathcal{L}_{43}$ by the
AMP interpretation of the graph in \cite[Fig.~8(a)]{andersson:2001}.

\subsection{All CI implications for four variables}
\label{sec:cis}

Although not pointed out explicitly, \cite{matus:2007} have
proved the following result via Theorem \ref{thm:matus}.

\begin{theorem}
  For $m=4$ variables, all the Gaussian CI implications follow from the
  implications (\ref{c3})-(\ref{eq:moreci5}) and the weak transitivity
  property in (\ref{eq:weaktran}).
\end{theorem}

Due to its disjunctive conclusion, the weak transitivity property is not a
CI implication in the sense of our strict Definition~\ref{def:CI-impl}.  A
natural problem is thus to find a set of CI implications in the sense of
this definition, from which all other such CI implications can be deduced.

Recall the last step of the search of complete relations in Section
\ref{sec:model}, which treats 94 semigaussoids that satisfy
(\ref{eq:moreci1})-(\ref{eq:moreci5}) but are not representable.  Of these,
46 are not complete and each yield new CI implications.  Namely, if
$\mathcal{L}$ is such a semi-gaussoid and $\bar{\mathcal{L}}$ the smallest
complete relation containing $\mathcal{L}$, then
$\mathcal{L}\Rightarrow(\bar{\mathcal{L}}\setminus\mathcal{L})$.  After a
careful check, we find that the following 13 CI implications together with
their duals generate all of the 46 CI implications given by the
non-complete semigaussoids:
\allowdisplaybreaks 
\begin{align}
  \label{eq:cistart}
   \{23|4, 23|14, 24|1, 34|1\} &\quad\Longrightarrow\quad \{23, 23|1,
   24|13, 34|12\}, \\ \label{eq:ci2} 
   \{23, 23|1, 24|1, 34|1\} &\quad\Longrightarrow\quad \{23|4, 23|14,
   24|13, 34|12\}, \\ \label{eq:ci3} 
   \{14|2, 14|3, 14|23, 23|14\} &\quad\Longrightarrow\quad \{14\}, \\
   \label{eq:ci4} 
   \{14, 14|2, 14|23, 23|14\} &\quad\Longrightarrow\quad \{14|3\} , \\ \label{eq:ci5}
   \{14, 14|2, 14|3, 23|14\} &\quad\Longrightarrow\quad \{14|23\} , \\ \label{eq:ci6}
   \{14, 14|23, 23|1,23|14\} &\quad\Longrightarrow\quad \{14|2, 14|3\} , \\
   \label{eq:ci7} 
   \{14|2, 14|3, 23|1, 23|14\} &\quad\Longrightarrow\quad \{14, 14|23\}
  , \\ \label{eq:ci8} 
   \{12, 14|3, 14|23, 23|14\} &\quad\Longrightarrow\quad \{12|3, 12|4,
   12|34, 23|4\} , \\ \label{eq:ci9} 
   \{12, 14|3, 23|4, 23|14\} &\quad\Longrightarrow\quad \{12|3, 12|4,
   12|34, 14|23\} , \\ \label{eq:ci10} 
   \{12|3, 14|2, 23|4, 23|14\} &\quad\Longrightarrow\quad \{12, 12|4,
   12|34, 14\} , \\ \label{eq:ci11} 
   \{12|3, 14, 14|2, 23|14\} &\quad\Longrightarrow\quad \{12, 12|4,
   12|34, 23|4\} , \\ \label{eq:ci12} 
   \{14|2, 23|1, 23|4, 23|14\} &\quad\Longrightarrow\quad \{23\} , \\
   \label{eq:ciend} 
   \{14|2, 23, 23|1, 23|14\} &\quad\Longrightarrow\quad \{23|4\} . 
 \end{align}
Although the CI implications (\ref{eq:cistart})-(\ref{eq:ciend}) are
  written with concrete choices of indices, they should be viewed as
  representing an equivalence class, that is, as the class of
  implications that can be obtained by a permutation of the indices.

\begin{theorem}
  \label{thm:ci}
  A relation on $[m]=[4]$ is complete if and only if it satisfies the
  CI implications (\ref{c3})-(\ref{eq:moreci5}),
  (\ref{eq:cistart})-(\ref{eq:ciend}) and the duals of
  (\ref{eq:cistart})-(\ref{eq:ciend}).
\end{theorem}

\begin{corollary}
  \label{thm:allci}
  All Gaussian CI implications for 4 variables can be deduced from
  (\ref{c3})-(\ref{eq:moreci5}),
  (\ref{eq:cistart})-(\ref{eq:ciend}) and the duals of
  (\ref{eq:cistart})-(\ref{eq:ciend}).
\end{corollary}

We conclude by demonstrating how to prove some of the implications in
(\ref{eq:cistart})-(\ref{eq:ciend}) by using the weak transitivity
property.  

\begin{proof}[Proof of (\ref{eq:cistart})]
  Suppose the CI statements in the relation on the left hand side are
  satisfied.  By (\ref{c3}),
  $\{24|1,34|1\}\Rightarrow\{24|13,34|12\}$.  By (\ref{c4}),
  $\{24|1,23|41\}\Rightarrow\{23|1,24|13\}$.  Hence, it remains to
  show that $23$ is implied.  By weak transitivity applied to $\{23|4,
  23|14\}$, we conclude that $12|4$ or $13|4$ hold.  There are two
  cases:
  \begin{enumerate}
  \item[(i)] Suppose that $12|4$ holds.  Then
    $\{12|4,24|1\}\Rightarrow \{12,24\}$ by the `intersection'
    implication in (\ref{c5}).  Applying~(\ref{c4}) again,
    $\{12,23|1\}\Rightarrow\{23\}$.
  \item[(ii)] By symmetry, we reach the conclusion that $23$ holds also
    when starting from $13|4$ instead of $12|4$.
    \qedhere
\end{enumerate}
\end{proof}

\begin{proof}[Proof of (\ref{eq:ci6})]
  In view of (\ref{eq:ci4}) (interchange 2 and 3 when necessary), we only
  need to show $14|2$ or $14|3$ must hold when the CI statements on the
  left hand side hold.  From weak transitivity applied to $\{23|1,23|14\}$,
  two cases arise:
  \begin{enumerate}
  \item[(i)] If $24|1$ holds then we may apply (\ref{c3}) to obtain
    $\{24|1,23|1\}\Rightarrow \{24|13\}$.  By (\ref{c5}), $\{24|13, 14|23\}
    \Rightarrow \{14|3\}$.
  \item[(ii)] If $34|1$ holds then we may conclude, by symmetry, that
    $14|2$ holds.
    \qedhere
  \end{enumerate}
\end{proof}

\begin{proof}[Proof of (\ref{eq:ci9})]
  It suffices to show that $12|3$ and $12|4$ are implied.  Applying weak
  transitivity to $\{23|4,23|14\}$ we can consider the two cases $12|4$ or
  $13|4$:
  \begin{itemize}
  \item [(i)] If $13|4$ holds then by (\ref{c5}), $\{13|4,14|3\} \Rightarrow
    \{13,14\}$.  Using (\ref{c3}), we can conclude $\{12,13,14\} \Rightarrow
    \{12|3, 12|4\}$.
  \item [(ii)] If $12|4$ holds then we can apply weak transitivity to
    $\{12,12|4\}$.  The two subcases are:
    \begin{enumerate}
    \item[(ii-a)] If $24$ holds then, by (\ref{c4}), $\{24,23|4\} \Rightarrow
      \{23\}$ and, by (\ref{c3}), $\{12,24\} \Rightarrow \{12|4\}$.  Another
      application of (\ref{c3}) yields $\{12,23\} \Rightarrow \{12|3\}$.
    \item[(ii-b)] If $14$ holds then we may apply weak transitivity to
      $\{14,14|3\}$ and split into two further subcases:
      \begin{enumerate}
      \item[(ii-b1)] If $13$ holds, then (\ref{c3}) yields $\{12,13\}\Rightarrow
        \{12|3\}$.  
      \item[(ii-b2)] If $34$ holds, then (\ref{c4}) yields
        $\{34,23|4\}\Rightarrow 
        \{23\}$.  Applying (\ref{c3}) we conclude $\{12,23\}\Rightarrow
        \{12|3\}$.  
        \qedhere
      \end{enumerate}
    \end{enumerate}
  \end{itemize}
\end{proof}

\begin{proof}[Proof of (\ref{eq:ci10})]
  Since $\{12,14|2\} \Rightarrow \{14,12|4\}$ by (\ref{c4}) and
  $\{12|4,23|4\} \Rightarrow \{12|34\}$ by (\ref{c3}), it suffices to show
  that $12$ holds under the assumed CI statements on the left hand side.
  Apply weak transitivity to $\{23|4, 23|14\}$ to obtain two cases:
  \begin{itemize}
  \item [(i)] If $12|4$ holds then (\ref{c5}) yields $\{12|4,14|2\}
    \Rightarrow \{12\}$.
  \item [(ii)]  If $13|4$ holds then (\ref{eq:moreci4}) yields
    $\{13|4,14|2,12|3\} \Rightarrow \{12\}$. 
    \qedhere
  \end{itemize}
\end{proof}

\clearpage


\begin{center}
\begin{longtable}{rlccl}
\caption[]{All representable relations on four variables, up to equivalence.}
\label{tab:rep} \\ 

\hline\hline $i$ & Elements of $\mathcal{L}_i$ & Singular &
Dual &\\ \hline  \endfirsthead

\hline\hline $i$ & Elements of $\mathcal{L}_i$ & Singular &
Dual &\\ \hline  \endhead

\multicolumn{5}{r}{{Continued on next page}} \\ \hline \endfoot

\hline \endlastfoot

1  & $12|\ast$, $13|\ast$, $14|\ast$, $23|\ast$, $24|\ast$, $34|\ast$
   && 1 \\[0.05cm]
2  & $13|\ast$, $14|\ast$, $23|\ast$, $24|\ast$, $34|\ast$ && 
   2 \\[0.05cm]
3  & $14|\ast$, $23|1$, $23|14$, $24|\ast$, $34|\ast$ &&
   38   \\[0.05cm]
4  & $13|\ast$, $14|\ast$, $23|\ast$ $24|\ast$ && 
   4 \\[0.05cm]
5  & $14|\ast$, $24|\ast$, $34|\ast$ && 
   5 \\[0.05cm]
6  & $23|1$, $23|14$, $24|1$, $24|13$, $34|1$, $34|12$ &&
   39\\[0.05cm]
7  & $14|2$, $14|23$, $23|1$, $23|14$, $34|1$, $34|2$, $34|12$ &&
   40\\[0.05cm]
8  & $23|1$, $23|14$, $34|1$, $34|12$ &&
   41\\[0.05cm]
9  & $23|1$, $23|14$, $24|3$, $34|1$, $34|12$ &&
   42\\[0.05cm]
10  & $12$, $12|3$, $23$, $23|1$, $23|14$, $34|1$, $34|12$ &&
   10 &(14)(23)\\[0.05cm]
11  & $23|1$, $23|14$, $24$, $34|1$, $34|12$ &&
   43 \\[0.05cm]
12  & $14|23$, $23|14$ &&
   44\\[0.05cm]
13  & $12$, $14|23$, $23|14$ &&
   45\\[0.05cm]
14  & $14$, $14|23$, $23|14$ &\checkmark&
   46\\[0.05cm]
15  & $14$, $14|23$, $23$, $23|14$ &\checkmark&
   15 \\[0.05cm]
16  & $23|14$ &&
   47\\[0.05cm]
17  & $12|3$, $23|14$ &&
   48 \\[0.05cm]
18  & $14|2$, $23|14$ &&
   49\\[0.05cm]
19  & $12|3$, $14|2$, $23|14$ &&
   50\\[0.05cm]
20  & $14|2$, $14|3$, $23|14$ &\checkmark&
   51\\[0.05cm]
21  & $13|2$, $23|14$, $24|3$ &&
   52\\[0.05cm]
22  & $12$, $23|14$ &&
   22 &(13)\\[0.05cm]
23  & $12$, $14|3$, $23|14$ &&
    25 & (13)\\[0.05cm]
24  & $12$, $23|14$, $24|3$ &\checkmark&
    24 & (13)\\[0.05cm]
25  & $12$, $23|14$, $34|2$ &&
    23 & (13)\\[0.05cm]
26  & $14$, $23|14$ &&
    26 &(12)(34)\\[0.05cm]
27  & $12|3$, $14$, $23|14$ &&
    27 &(12)(34)\\[0.05cm]
28  & $13|2$, $14$, $23|14$, $24|3$ &\checkmark&
    28 &(13)(24) \\[0.05cm]
29  & $23$, $23|14$ &\checkmark&
    29\\[0.05cm]
30  & $14|2$, $23$, $23|14$ &\checkmark&
    30 & (23)\\[0.05cm]
31  & $12|3$ &&
    31 &(34)\\[0.05cm]
32  & $12|3$, $12|4$ &\checkmark&
    32\\[0.05cm]
33  & $12|3$, $13|4$ &&
    33 &(23)\\[0.05cm]
34  & $12|3$, $14|2$, $23|4$ &&
    34 & (12)(34)\\[0.05cm]
35  & $12|3$, $34|1$ &&
    35 &(12)(34)\\[0.05cm] 
36  & $12|3$, $12|4$, $34|1$ &\checkmark&
    36 &(12)\\[0.05cm]
37  & $12|3$, $12|4$, $34|1$, $34|2$ &\checkmark& 37\\[0.05cm]
38  & $14|\ast$, $23$, $23|4$, $24|\ast$, $34|\ast$ && 3\\[0.05cm]
39  & $23$, $23|4$, $24$, $24|3$, $34$, $34|2$ && 6\\[0.05cm]
40  & $14$, $14|3$, $23$, $23|4$, $34$, $34|1$, $34|2$ && 7\\[0.05cm]
41  & $23$, $23|4$, $34$, $34|2$ && 8 \\[0.05cm]
42  & $23$, $23|4$, $24|1$, $34$, $34|2$ && 9\\[0.05cm]
43  & $23$, $23|4$, $24|13$, $34$, $34|2$ &&11 \\[0.05cm]
44  & $14$, $23$ && 12\\[0.05cm]
45  & $12|34$, $14$, $23$ && 13\\[0.05cm]
46  & $14$, $14|23$, $23$ &\checkmark& 14\\[0.05cm]
47  & $23$ && 16\\[0.05cm]
48  & $12|4$, $23$ && 17\\[0.05cm]
49  & $14|3$, $23$ && 18\\[0.05cm]
50  & $12|4$, $14|3$, $23$ && 19\\[0.05cm]
51  & $14|2$, $14|3$, $23$ &\checkmark& 20\\[0.05cm]
52  & $13|4$, $23$, $24|1$ && 21\\[0.05cm]
53 & $\emptyset$  &&   53\\[0.05cm]
\hline
\end{longtable}
\end{center}


\begin{center}
\begin{longtable}[p]{rl@{\hspace{.75cm}}rl}
\caption{All complete non-representable relations on four variables, up to
  equivalence.} \label{tab:com}
\\

\hline\hline $i$ & Representable decomposition of $\mathcal{L}_i$ &
\multicolumn{2}{c}{$\!\!\!\!\!\!$Equivalence class}\\ 
\hline \endfirsthead 

\hline\hline $i$ & Representable decomposition of $\mathcal{L}_i$ &
\multicolumn{2}{c}{$\!\!\!\!\!\!$Equivalence class}\\ \hline \endhead 

\multicolumn{4}{r}{{Continued on next
    page}} \\ \hline \endfoot \hline \endlastfoot

54  & $13|\ast$, $14|\ast$, $23|\ast$, $24|\ast$, $34|\ast$  &  2 &    \\ 
 & $12|\ast$, $14|\ast$, $23|\ast$, $24|\ast$, $34|\ast$  &  2   & (234)  \\ 
\hline 
55  & $13|\ast$, $23|\ast$, $24|1$, $24|13$, $34|\ast$  &  3  & (34)  \\ 
 & $12|\ast$, $23|\ast$, $24|\ast$, $34|1$, $34|12$  &  3   &(234)  \\ 
\hline 
56  & $13|\ast$, $14|\ast$, $23|\ast$, $24|\ast$, $34|\ast$  &  2  &   \\ 
 & $12|\ast$, $23|\ast$, $24|\ast$, $34|1$, $34|12$  &  3   &(234)  \\ 
\hline 
57  & $13|\ast$, $14|\ast$, $23|\ast$, $24|\ast$, $34|\ast$  &  2     \\ 
 & $12|\ast$, $14|\ast$, $23|\ast$, $24|\ast$, $34|\ast$  &  2   &(234)  \\ 
 & $12|\ast$, $13|\ast$, $23|\ast$, $24|\ast$, $34|\ast$  &  2   &(24)  \\ 
\hline 
58  & $14|\ast$, $23|1$, $23|14$, $24|\ast$, $34|\ast$  &  3     \\ 
 & $13|\ast$, $14|2$, $14|23$, $23|\ast$, $34|\ast$  &  3   &(1243)  \\ 
 & $12|\ast$, $14|\ast$, $23|\ast$, $34|\ast$  &  4   &(234)  \\ 
\hline 
59  & $14|\ast$, $23|1$, $23|14$, $24|\ast$, $34|\ast$  &  3     \\ 
 & $12|\ast$, $14|\ast$, $23|\ast$, $34|\ast$  &  4   &(234)  \\ 
\hline 
60  & $13|\ast$, $14|\ast$, $23|\ast$, $24|\ast$, $34|\ast$  &  2     \\ 
 & $12|\ast$, $14|\ast$, $23|\ast$, $34|\ast$  &  4   &(234)  \\ 
\hline 
61  & $13|\ast$, $23|\ast$, $34|\ast$  &  5   &(34)  \\ 
 & $12|4$, $12|34$, $23|1$, $23|4$, $23|14$, $34|1$, $34|12$  &  7   &(24)  \\ 
\hline 
62  & $13|\ast$, $23|\ast$, $34|\ast$  &  5   &(34)  \\ 
 & $12$, $12|3$, $23$, $23|1$, $23|14$, $34|1$, $34|12$  &  10 &    \\ 
\hline
\pagebreak 
63  & $12|\ast$, $23|\ast$, $24|\ast$, $34|1$, $34|12$  &  3   &(234)  \\ 
 & $12|\ast$, $14|\ast$, $23|\ast$, $34|\ast$  &  4   &(234)  \\ 
 & $13|\ast$, $23|\ast$, $34|\ast$  &  5   &(34)  \\ 
\hline 
64  & $12|\ast$, $23|\ast$, $24|\ast$, $34|1$, $34|12$  &  3   &(234)  \\ 
 & $12|\ast$, $14|\ast$, $23|\ast$, $34|\ast$  &  4   &(234)  \\ 
 & $12$, $12|3$, $13|\ast$, $23|\ast$, $34|\ast$  &  38   &(143)  \\ 
\hline 
65  & $14|\ast$, $23|1$, $23|14$, $24|\ast$, $34|\ast$  &  3     \\ 
 & $12|\ast$, $23|\ast$, $24|\ast$, $34|1$, $34|12$  &  3   &(234)  \\ 
 & $13|\ast$, $23|\ast$, $24$, $24|3$, $34|\ast$  &  38   &(34)  \\ 
\hline 
66  & $12|\ast$, $23|\ast$, $24|\ast$, $34|1$, $34|12$  &  3   &(234)  \\ 
 & $13|\ast$, $23|\ast$, $24$, $24|3$, $34|\ast$  &  38   &(34)  \\ 
\hline 
67  & $12|\ast$, $14|\ast$, $23|\ast$, $34|\ast$  &  4   &(234)  \\ 
 & $13|\ast$, $23|\ast$, $34|\ast$  &  5   &(34)  \\ 
\hline 
68  & $14|2$, $14|23$, $23|1$, $23|14$, $34|1$, $34|2$, $34|12$  &  7     \\ 
 & $13|2$, $13|4$, $13|24$, $14|2$, $14|23$, $23|4$, $23|14$  &  7   &(1243)  \\ 
\hline 
69  & $13|\ast$, $14|\ast$, $23|\ast$, $24|\ast$  &  4     \\ 
 & $14|2$, $14|23$, $23|1$, $23|14$, $34|1$, $34|2$, $34|12$  &  7 &    \\ 
\hline 
70  & $14|\ast$, $23|1$, $23|14$, $24|\ast$, $34|\ast$  &  3     \\ 
 & $12|\ast$, $13|\ast$, $14|\ast$, $23|4$, $23|14$  &  3   &(14)(23)  \\ 
 & $13|\ast$, $14|\ast$, $23|\ast$, $24|\ast$  &  4     \\ 
 & $12|\ast$, $14|\ast$, $23|\ast$, $34|\ast$  &  4   &(234)  \\ 
\hline 
71  & $14|\ast$, $23|1$, $23|14$, $24|\ast$, $34|\ast$  &  3     \\ 
 & $13|\ast$, $14|\ast$, $23|\ast$, $24|\ast$  &  4     \\ 
 & $12|\ast$, $14|\ast$, $23|\ast$, $34|\ast$  &  4   &(234)  \\ 
\hline 
72  & $13|\ast$, $14|\ast$, $23|\ast$, $24|\ast$  &  4     \\ 
 & $12|\ast$, $14|\ast$, $23|\ast$, $34|\ast$  &  4   &(234)  \\ 
\hline 
73  & $23|1$, $23|14$, $34|1$, $34|12$  &  8     \\ 
 & $23|1$, $23|14$, $24|1$, $24|13$  &  8   &(243)  \\ 
\hline 
74  & $14|2$, $14|23$, $23|1$, $23|14$, $34|1$, $34|2$, $34|12$  &  7     \\ 
 & $14$, $14|2$, $23|1$, $23|14$, $24$, $24|1$, $24|13$  &  10   &(243)  \\ 
\hline 
75  & $12|\ast$, $23|\ast$, $24|\ast$  &  5   &(234)  \\ 
 & $12|3$, $12|4$, $12|34$, $14|3$, $14|23$, $23|4$, $23|14$  &  7   &(13)(24)  \\ 
 & $12|3$, $13|4$, $13|24$, $23|4$, $23|14$  &  9   &(142)  \\ 
\hline 
76  & $13|\ast$, $23|\ast$, $34|\ast$  &  5   &(34)  \\ 
 & $12|\ast$, $23|\ast$, $24|\ast$  &  5   &(234)  \\ 
 & $12|4$, $12|34$, $23|1$, $23|4$, $23|14$, $34|1$, $34|12$  &  7   &(24)  \\ 
 & $13|4$, $13|24$, $23|1$, $23|4$, $23|14$, $24|1$, $24|13$  &  7   &(243)  \\ 
\hline 
77  & $12|\ast$, $23|\ast$, $24|\ast$  &  5   &(234)  \\ 
 & $12$, $12|3$, $23$, $23|1$, $23|14$, $34|1$, $34|12$  &  10     \\ 
 & $12$, $12|3$, $13$, $13|2$, $23|14$  &  43   &(1423)  \\ 
\hline 
78  & $12|\ast$, $23|\ast$, $24|\ast$  &  5   &(234)  \\ 
 & $12$, $12|4$, $12|34$, $14$, $14|2$, $23|4$, $23|14$  &  10   &(1432)  \\ 
 & $12$, $13|4$, $13|24$, $23|4$, $23|14$  &  11   &(142)  \\ 
\hline 
\pagebreak 
79  & $12|\ast$, $13|\ast$, $14|\ast$, $23|4$, $23|14$  &  3   &(14)(23)  \\ 
 & $12|\ast$, $14|\ast$, $23|\ast$, $34|\ast$  &  4   &(234)  \\ 
 & $12|\ast$, $23|\ast$, $24|\ast$  &  5   &(234)  \\ 
 & $12$, $12|3$, $13|\ast$, $23|\ast$, $34|\ast$  &  38   &(143)  \\ 
\hline 
80  & $14$, $14|2$, $23|1$, $23|14$, $24$, $24|1$, $24|13$  &  10   &(243)  \\ 
 & $12$, $12|4$, $12|34$, $14$, $14|2$, $23|4$, $23|14$  &  10   &(1432)  \\ 
\hline 
81  & $14$, $14|3$, $23|1$, $23|14$, $34$, $34|1$, $34|12$  &  10   &(24)  \\ 
 & $14$, $14|2$, $23|1$, $23|14$, $24$, $24|1$, $24|13$  &  10   &(243)  \\ 
\hline 
82  & $12|\ast$, $14|\ast$, $23|\ast$, $34|\ast$  &  4   &(234)  \\ 
 & $14$, $14|2$, $23|1$, $23|14$, $24$, $24|1$, $24|13$  &  10   &(243)  \\ 
\hline 
83  & $13|\ast$, $23|\ast$, $34|\ast$  &  5   &(34)  \\ 
 & $12|\ast$, $23|\ast$, $24|\ast$  &  5   &(234)  \\ 
 & $12$, $12|3$, $23$, $23|1$, $23|14$, $34|1$, $34|12$  &  10     \\ 
 & $13$, $13|2$, $23$, $23|1$, $23|14$, $24|1$, $24|13$  &  10   &(23)  \\ 
\hline 
84  & $13|\ast$, $14|\ast$, $23|\ast$, $24|\ast$  &  4     \\ 
 & $12|\ast$, $14|\ast$, $23|\ast$, $34|\ast$  &  4   &(234)  \\ 
 & $13|\ast$, $23|\ast$, $34|\ast$  &  5   &(34)  \\ 
 & $12|\ast$, $23|\ast$, $24|\ast$  &  5   &(234)  \\ 
\hline 
85  & $13|\ast$, $14|2$, $14|23$, $23|\ast$, $34|\ast$  &  3   &(1243)  \\ 
 & $13|\ast$, $14|\ast$, $23|\ast$, $24|\ast$  &  4     \\ 
 & $12|\ast$, $14|\ast$, $23|\ast$, $34|\ast$  &  4   &(234)  \\ 
 & $12|\ast$, $14$, $14|2$, $23|\ast$, $24|\ast$  &  38   &(13)(24)  \\ 
\hline 
86  & $13|\ast$, $23|\ast$, $24$, $24|3$, $34|\ast$  &  38   &(34)  \\ 
 & $12|\ast$, $23|\ast$, $24|\ast$, $34$, $34|2$  &  38   &(234)  \\ 
\hline 
87  & $13|\ast$, $14|\ast$, $23|\ast$, $24|\ast$, $34|\ast$  &  2     \\ 
 & $12|\ast$, $23|\ast$, $24|\ast$, $34$, $34|2$  &  38   &(234)  \\ 
\hline 
88  & $12|\ast$, $14|\ast$, $23|\ast$, $34|\ast$  &  4   &(234)  \\ 
 & $14|\ast$, $23$, $23|4$, $24|\ast$, $34|\ast$  &  38     \\ 
 & $13|\ast$, $14$, $14|3$, $23|\ast$, $34|\ast$  &  38   &(1243)  \\ 
\hline 
89  & $12|\ast$, $14|\ast$, $23|\ast$, $34|\ast$  &  4   &(234)  \\ 
 & $14|\ast$, $23$, $23|4$, $24|\ast$, $34|\ast$  &  38     \\ 
\hline 
90  & $13|\ast$, $23|\ast$, $34|\ast$  &  5   &(34)  \\ 
 & $12$, $12|3$, $23$, $23|1$, $23|4$, $34$, $34|2$  &  40   &(24)  \\ 
\hline 
91  & $13|\ast$, $23|\ast$, $34|\ast$  &  5   &(34)  \\ 
 & $12|4$, $12|34$, $23$, $23|4$, $23|14$, $34$, $34|2$  &  10   &(14)(23)  \\ 
\hline 
92  & $12|\ast$, $14|\ast$, $23|\ast$, $34|\ast$  &  4   &(234)  \\ 
 & $13|\ast$, $23|\ast$, $34|\ast$  &  5   &(34)  \\ 
 & $12|\ast$, $23|\ast$, $24|\ast$, $34$, $34|2$  &  38   &(234)  \\ 
\hline 
93  & $13|\ast$, $23|\ast$, $24|1$, $24|13$, $34|\ast$  &  3   &(34)  \\ 
 & $14|\ast$, $23$, $23|4$, $24|\ast$, $34|\ast$  &  38     \\ 
 & $12|\ast$, $23|\ast$, $24|\ast$, $34$, $34|2$  &  38   &(234)  \\ 
\hline 
94  & $14$, $14|3$, $23$, $23|4$, $34$, $34|1$, $34|2$  &  40     \\ 
 & $13$, $13|2$, $13|4$, $14$, $14|3$, $23$, $23|1$  &  40   &(1243)  \\ 
\hline 
95  & $13|\ast$, $14|\ast$, $23|\ast$, $24|\ast$  &  4     \\ 
 & $14$, $14|3$, $23$, $23|4$, $34$, $34|1$, $34|2$  &  40     \\ 
\hline 
\pagebreak 
96  & $13|\ast$, $14|\ast$, $23|\ast$, $24|\ast$  &  4     \\ 
 & $12|\ast$, $14|\ast$, $23|\ast$, $34|\ast$  &  4   &(234)  \\ 
 & $14|\ast$, $23$, $23|4$, $24|\ast$, $34|\ast$  &  38     \\ 
 & $12|\ast$, $13|\ast$, $14|\ast$, $23$, $23|1$  &  38   &(14)(23)  \\ 
\hline 
97  & $13|\ast$, $14|\ast$, $23|\ast$, $24|\ast$  &  4     \\ 
 & $12|\ast$, $14|\ast$, $23|\ast$, $34|\ast$  &  4   &(234)  \\ 
 & $14|\ast$, $23$, $23|4$, $24|\ast$, $34|\ast$  &  38     \\ 
\hline 
98  & $23$, $23|4$, $34$, $34|2$  &  41     \\ 
 & $23$, $23|4$, $24$, $24|3$  &  41   &(243)  \\ 
\hline 
99  & $14|3$, $14|23$, $23$, $23|4$, $24$, $24|3$, $24|13$  &  10   &(134)  \\ 
 & $14$, $14|3$, $23$, $23|4$, $34$, $34|1$, $34|2$  &  40     \\ 
\hline 
100  & $12|\ast$, $23|\ast$, $24|\ast$  &  5   &(234)  \\ 
 & $12$, $12|3$, $12|4$, $14$, $14|2$, $23$, $23|1$  &  40   &(13)(24)  \\ 
 & $12|4$, $13$, $13|2$, $23$, $23|1$  &  42   &(142)  \\ 
\hline 
101  & $13|\ast$, $23|\ast$, $34|\ast$  &  5   &(34)  \\ 
 & $12|\ast$, $23|\ast$, $24|\ast$  &  5   &(234)  \\ 
 & $12$, $12|3$, $23$, $23|1$, $23|4$, $34$, $34|2$  &  40   &(24)  \\ 
 & $13$, $13|2$, $23$, $23|1$, $23|4$, $24$, $24|3$  &  40   &(243)  \\
\end{longtable}
\end{center}

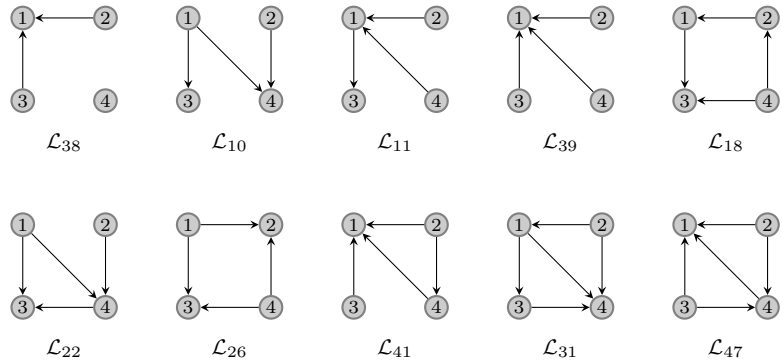
\begin{figure}[h]
  \centering
  \vspace{0.3cm}
\begin{tikzpicture}[>=stealth,scale=1.1]

\node (1) at ( 0,3.5) [usual] {\scriptsize 1};
\node (2) at ( 1,3.5) [usual] {\scriptsize 2}
  edge [->] (1);  
\node (3) at ( 0,2.5) [usual] {\scriptsize 3}
  edge [->] (1);  
\node (4) at ( 1,2.5) [usual] {\scriptsize 4};

\node (13) at ( 2,3.5) [usual] {\scriptsize 1};
\node (14) at ( 3,3.5) [usual] {\scriptsize 2};
\node (15) at ( 2,2.5) [usual] {\scriptsize 3}
    edge [<-] (13);
\node (16) at ( 3,2.5) [usual] {\scriptsize 4}
    edge [<-] (13)
    edge [<-] (14);

\node (5) at ( 6,3.5) [usual] {\scriptsize 1};
\node (6) at ( 7,3.5) [usual] {\scriptsize 2}
  edge [->] (5);
\node (7) at ( 6,2.5) [usual] {\scriptsize 3}
  edge [->] (5);
\node (8) at ( 7,2.5) [usual] {\scriptsize 4}
  edge [->] (5);

\node (9)  at ( 4,3.5) [usual] {\scriptsize 1};
\node (10) at ( 5,3.5) [usual] {\scriptsize 2}
  edge [->] (9);
\node (11) at ( 4,2.5) [usual] {\scriptsize 3}
  edge [<-] (9);
\node (12) at ( 5,2.5) [usual] {\scriptsize 4}
  edge [->] (9);

\node (23) at ( 8,2.5) [usual] {\scriptsize 3};
\node (24) at ( 9,2.5) [usual] {\scriptsize 4}
    edge [->] (23);
\node (22) at ( 9,3.5) [usual] {\scriptsize 2}
    edge [<-] (24);
\node (21) at ( 8,3.5) [usual] {\scriptsize 1}
    edge [<-] (22)
    edge [->] (23);

\node  at (0.5,2)  {\footnotesize $\mathcal{L}_{38}$};
\node  at (2.5,2)  {\footnotesize $\mathcal{L}_{10}$};
\node  at (4.5,2)  {\footnotesize $\mathcal{L}_{11}$};
\node  at (6.5,2)  {\footnotesize $\mathcal{L}_{39}$};
\node  at (8.5,2)  {\footnotesize $\mathcal{L}_{18}$};


\node (29) at ( 0,1) [usual] {\scriptsize 1};
\node (30) at ( 1,1) [usual] {\scriptsize 2};
\node (31) at ( 0,0) [usual] {\scriptsize 3}
    edge [<-] (29);
\node (32) at ( 1,0) [usual] {\scriptsize 4}
    edge [<-] (29)
    edge [<-] (30)
    edge [->] (31);

\node (18) at ( 3,1) [usual] {\scriptsize 2};
\node (19) at ( 2,0) [usual] {\scriptsize 3};
\node (20) at ( 3,0) [usual] {\scriptsize 4}
    edge [->] (18)
    edge [->] (19);
\node (17) at ( 2,1) [usual] {\scriptsize 1}
    edge [->] (18)
    edge [->] (19);

\node (25) at ( 4,1) [usual] {\scriptsize 1};
\node (26) at ( 5,1) [usual] {\scriptsize 2}
    edge [->] (25);
\node (27) at ( 4,0) [usual] {\scriptsize 3}
    edge [->] (25);
\node (28) at ( 5,0) [usual] {\scriptsize 4}
    edge [->] (25)
    edge [<-] (26);

\node (37) at ( 6,1) [usual] {\scriptsize 1};
\node (38) at ( 7,1) [usual] {\scriptsize 2}
    edge [->] (37);
\node (39) at ( 6,0) [usual] {\scriptsize 3}
    edge [<-] (37);
\node (40) at ( 7,0) [usual] {\scriptsize 4}
    edge [<-] (37)
    edge [<-] (38)
    edge [<-] (39);

\node (33) at ( 8,1) [usual] {\scriptsize 1};
\node (34) at ( 9,1) [usual] {\scriptsize 2}
    edge [->] (33);
\node (35) at ( 8,0) [usual] {\scriptsize 3}
    edge [->] (33);
\node (36) at ( 9,0) [usual] {\scriptsize 4}
    edge [->] (33)
    edge [<-] (34)
    edge [<-] (35);

\node  at (0.5,-0.5)  {\footnotesize $\mathcal{L}_{22}$};
\node  at (2.5,-0.5)  {\footnotesize $\mathcal{L}_{26}$};
\node  at (4.5,-0.5)  {\footnotesize $\mathcal{L}_{41}$};
\node  at (6.5,-0.5)  {\footnotesize $\mathcal{L}_{31}$};
\node  at (8.5,-0.5)  {\footnotesize $\mathcal{L}_{47}$};

\end{tikzpicture}
  \caption{Representable relations associated with acyclic digraphs.
    The relations are labelled in reference to Table~\ref{tab:rep}.}
  \label{fig:dags}
\end{figure}

\end{appendix}


\bibliographystyle{amsalpha}
\bibliography{drtonxiao-ams}

\end{document}